\documentclass[a4paper,12pt]{article}
\usepackage[heightrounded]{geometry}

\usepackage{geometry}
\usepackage{fancyhdr} \usepackage{preambulaFilipa}

\title{Tensor rank of the direct sum of two copies of $2 \times 2$ matrix multiplication tensor is 14.}
\author{Filip Rupniewski\footnote{\nolinkurl{filip.rupniewski@unibe.ch}, 
Mathematical Institute, University of Bern, Alpeneggstrasse 22, 3012 Bern, Switzerland}}
\date{\today}

\newcommand*{\fullref}[1]{\hyperref[{#1}]{\autoref*{#1}: \nameref*{#1}}}

\begin{document}

\allowhyphens
\selectlanguage{english}

\maketitle

\begin{abstract}
The article is concerned with the problem of the additivity of the tensor rank.
That is for two independent tensors we study when the rank of their direct sum is equal to the sum of their individual ranks.
The statement saying that additivity always holds was previously known as Strassen's conjecture (1969) 
until Shitov proposed counterexamples (2019). They are not explicit and only known to exist asymptotically for very large tensor spaces.
In this article, we show that for some small three-way tensors the additivity holds.
For instance, we give a proof that another conjecture stated by Strassen (1969) is true. It is the particular case of the general Strassen's additivity conjecture where tensors are a pair of $2 \times 2$ matrix multiplication tensors. 
In addition, we show that the Alexeev-Forbes-Tsimerman substitution method preserves the structure of a direct sum of tensors.
\end{abstract}

\medskip
{\footnotesize

\section*{Acknowledgments}

First of all, I would like to express my deep gratitude to Jarosław Buczyński, for the introduction to the topic and numerous discussions and explanations. 

I would like to thank 
Maciej Gałązka,
Joachim Jelisiejew, 
Joseph Landsberg, 
Tomasz Mańdziuk,
Mateusz Micha{\l}ek and 
Elisa Postinghel
for their valuable comments.

Finally, I would like to thank the National Science Center for supporting my research (project number 2019/33/N/ST1/00068).

\addtolength{\headheight}{6mm}

\section{Introduction}\label{ch:introduction}
\begin{notation}\label{n:ABCabc}
Let $A, A', A'', B, B', B'', C, C', C''$ be finitely dimensional vector spaces over a field $\kk$ of dimensions $\bfa, \bfa',\bfa'',\bfb, \bfb',\bfb'',\bfc, \bfc',\bfc''$ respectively.  
\end{notation}

We consider order 3 tensors, i.e. tensors $p \in A \otimes B \otimes C$. 
The classical question 
is how complicated they are, i.e. what is their tensor rank?

\begin{definition}\label{d:3way_tensor_rank}
 The subset $W \subseteq A \otimes B \otimes C $ is of rank $R(W)=r$ if and only if $r$ is a minimal number such that there exists a tensor decomposition of length $r$, i.e. there exist $\{a_1,a_2,...,a_{r}\} \subseteq A$, $\{b_1,b_2,...,b_{r}\} \subseteq  B$ and $\{c_1,c_2,...,c_{r} \} \subseteq C$ such that $W$ belongs to the linear span $\langle a_1 \otimes b_1 \otimes c_1, a_2 \otimes b_2 \otimes c_2, \ldots, a_r \otimes b_r \otimes c_r \rangle$. In case $W$ consist of just one tensor, we recover the classical definition of tensor rank.
 Tensors of rank 1 are called \emph{simple tensors}.
\end{definition}

The tensor decomposition, also known as \emph{Canonical Polyadic Decomposition} and \emph{CANDECOMP/PARAFAC (CP)} tensor decomposition, can be considered to be higher order generalizations of the matrix singular value decomposition (SVD).  In the analogy to the analyzing complicated data coming from physical world, the rank should correspond to the number of simple ingredients affecting our complicated state.

The interest in the subject has expanded to other fields. Examples include signal processing \cite{signal_processing}, numerical linear algebra \cite{numerical_linear_algebra}, coding theory \cite{BNRS19}, computer vision \cite{computer_vision}, numerical analysis \cite{numerical_analysis}, data mining \cite{data_mining}, graph analysis \cite{graph_analysis}, neuroscience \cite{neuroscience}, and more. 
Tensors and tensor rank appear also in physics. 
In quantum mechanics, the rank of a tensor is a measure of degree of entanglement. 
Detailed introduction to problems of pure states entanglement and connection with variants of tensor rank is contained in \cite{BFZ20}.
More about tensor decomposition one can read in the survey \cite{Kolda-survey}. 
Motivations to study tensor rank are included in for instance \cite{comon_tensor_decompositions_survey}, \cite{landsberg_tensorbook},
   \cite{carlini_grieve_oeding_four_lectures_on_secants} and references therein.

For a matrix, i.e. an order two tensor, the rank of a tensor and rank of a matrix coincide. 
The computation of the matrix rank is usually obtained by applying the Gaussian elimination process. 
A classical result says that it is  computable in a polynomial time \cite[p.12]{farebrother}.
 There are numerical algorithms to decompose a tensor \cite{Tlab3deco, TTB}. However, 
 in contrast to matrix rank, there is no effective algorithm calculating the rank of a given tensor. Hastad \cite{Hastad} proved that the tensor rank is NP-hard to compute.
Since an upper bound is obtained by exhibiting a decomposition, the main challenge is to find lower bounds for the rank. 

One of known methods which apply in general is the Alexeev-Forbes-Tsimerman substitution method \cite{alexeev_forbes_tsimerman_Tensor_rank_some_lower_and_upper_bounds}. 
It gives a lower bound for $R(p)$ by a rank of another tensor $\tilde{p}$. 
For coordinate free rephrasement see \cite{BPR} and Propositions \ref{proposition_for_AFT_method_coordinate_free}, \ref{proposition_for_AFT_method_slice_A}. 
After fixing a basis for each vector spaces, the tensor $p$ is a 3-dimensional array. One obtains $\tilde{p}$ from $p$ after distinguishing a slice $M$ of $p$ and adding $w_i M$ to the $i$-th slice of $p$, where $w_i \in \kk$ are some uknown weights.
In addition, if the choosen slice is of rank one then the method gives the precise difference between ranks of both tensors. 
We show that in this case also the structure of a direct sum of two tensors is preserved and the „hook” structure of tensors does not change as well (see Proposition \ref{proposition_for_AFT_method_slice_A_improved} and Corollary \ref{c:removing_rank_1_slice}).

 Our main interest is the question about additivity of tensors.  More precisely, given two tensors from independent tensor spaces, when is the rank of their direct sum equal the sum of their ranks?
 The statement of the Strassen's conjecture was that additivity always holds
 \cite{strassen_vermeidung_von_divisionen}, 
 Shitov \cite{shitov_counter_example_to_Strassen} disproved it, but did not give an explicit counterexample. One way to find an explicit one is to analyze small dimensional cases. 

 In this article we give some sufficient conditions for additivity and improve results from \cite{BPR}. The following theorem summarizes our main outcomes.
 \begin{thm}\label{thm_additivity_rank_intro_filip}
 With Notation \ref{n:ABCabc},
  assume $p' \in A' \otimes B' \otimes C'$ and $p'' \in A'' \otimes B'' \otimes C''$ and let
  \[
      p = p' \oplus p'' \in (A'\oplus A'') \otimes (B'\oplus B'') \otimes(C'\oplus C'').
  \]
  If at least one of the following conditions holds, 
     then the additivity of the rank holds for $p$, that is $R(p) = R(p') + R(p'')$:
  \begin{enumerate}
   \item $p''\in A'' \otimes (B''\otimes \kk^1 + \kk^2 \otimes C'')$ (this part of the statement is valid for any base field $\kk$).
   \label{it:4_thm_additivity_rank_intro_filip}
    \item $\kk=\CC$ and the pair $((\bfa', \bfb', \bfc'), (\bfa'', \bfb'', \bfc''))$ equals either $((4,4,3),(4,4,3))$ or $((4,4,3),(4,3,4))$,
    \label{it:5_thm_additivity_rank_intro_filip}
    \item $\kk=\CC$ and 
both tensors have ranks less or equal 7. 
    In particular, $R(\mu_{2,2,2} \oplus \mu_{2,2,2})=R(\mu_{2,2,2}) + R(\mu_{2,2,2})$, where $\mu_{2,2,2}$ denotes the $2 \times 2$ matrix multiplication tensor.
    \label{it:6_thm_additivity_rank_intro_filip}
\end{enumerate}
  Analogous statements hold if we exchange the roles of $A$, $B$, $C$ and/or of ${}'$ and ${}''$. 
\end{thm}

 Following \cite{BPR}, we distinguish seven types of matrices from a minimal decomposition and show that to prove the additivity of the tensor rank, one can get rid of two of those types. Buczyński, Postinghel and Rupniewski called the process ``repletion and digestion''. 
   In other words, there is a smaller example, a pair $(\tilde{p_1}, \tilde{p_2})$ without those two types in its minimal decomposition. If the additivity property holds for $(\tilde{p_1}, \tilde{p_2})$, then it also holds for the original pair. 
This was the core observation, which let to prove one of the main results of the article \cite{BPR} (see Theorem~\ref{thm_additivity_rank_intro_jarek}). In this article we review the process and ``replete and digest'' with respect to one distinguished rank one matrix in place of the family of matrices at once (Sections \ref{ss:replete_pairs} and \ref{ss:digestion}). It gives more control over the structure of the outcome. 
 In result, 
    we are able to improve the substitution method and prove additivity in cases written above (Theorem \ref{thm_additivity_rank_intro_filip}).
    In particular, point \ref{it:4_thm_additivity_rank_intro_filip} generalize the proposition of Buczyński, Postinghel and Rupniewski \cite[Prop. 3.17]{BPR} to an arbitrary base field. The proof itself is also easier and shorter.

 It let us 
 answers another open question given in 1969 by Strassen \cite{strassen_vermeidung_von_divisionen} about additivity of a pair of $2 \times 2$ matrix multiplication tensors. We prove that given arbitrary four matrices  $\{ M',N',M'',N'' \} \subseteq \kk^{2 \times 2}$ there is no faster way to calculate both products $M'N'$ and $M''N''$ simultaneously, than doing it independently --- calculating $M'N'$ first and $M''N''$ afterwords.

\section{Overview}\label{s:overview}
Section \ref{s:preliminaries} contains motivation and introduction to the problem of additivity of tensor rank. Following sections (Section \ref{s:ranks_and_slices} and \ref{s:dir_sums}) contains definitions and tools
needed to prove main claims in Section \ref{s:rank_one_matrices_and_additive_rank}. 
In particular, one can find there connection of tensors with the space of matrices, definition of \emph{``hook''-shaped space} and \emph{substitution method}.
To give more information of the structure of the output of the substitution method and to give conditions for the additivity, there is a necessity of the analysis of slices of $(p_1 +p_2)((A'\oplus A'')^*)$. It is located in Section~\ref{s:rank_one_matrices_and_additive_rank} together with introduction of ``repletion'' and ``digestion'' processes with respect to one matrix. In the same section one can find proofs of main theorems.
In particular, there is shown why a pair of $2 \times 2$ matrix multiplication tensor has the rank additivity property, i.e.  $R(\mu_{2,2,2} \oplus \mu_{2,2,2})=R(\mu_{2,2,2}) + R(\mu_{2,2,2})$.

\section{Preliminaries}\label{s:preliminaries}
\subsection{Matrix multiplication}
The standard way to calculate the product of two $2\times2$ matrices
%
uses 8 multiplications and 4 additions. 
In 1969, Strassen presented an algorithm 
using 18 additions, but only 7 multiplications. 
\begin{thm}[Strassen's algorithm for multiplication of two $2 \times 2$ matrices, \cite{strassen_gaussian_elimination_is_not_optimal}]\label{t:strassen_multiplication}
Let $M = (a_{i,j} )$ and $N = (b_{i,j})$ be two $2 \times 2$ matrices and $M N = (c_{i,j})$ be their product. Then calculating 7 products (of numbers):
\begin{align*}
I :=& (a_{1,1} + a_{2,2} )(b_{1,1} + b_{2,2} ); &
II :=& (a_{2,1} + a_{2,2} )b_{1,1} ; &
III :=& a_{1,1} (b_{1,2} - b_{2,2} ); \\
IV :=& a_{2,2} (-b_{1,1} + b_{2,1} ); &
V :=& (a_{1,1} + a_{1,2} )b_{2,2} ; &
VI :=& (-a_{1,1} + a_{2,1} )(b_{1,1} + b_{1,2} );\\
VII :=& (a_{1,2} - a_{2,2} )(b_{2,1} + b_{2,2} );  &
\end{align*}
we can present $(c_{i,j})$ using just their sums:
\[
\begin{bmatrix}
 c_{1,1} & c_{1,2}\\
 c_{2,1} & c_{2,2}
\end{bmatrix}
=
\begin{bmatrix}
 I + IV - V + VII & II + IV\\
 III + V & I + III - II + VI
\end{bmatrix}
\]
\end{thm}

Since the multiplication of two $n \times n$ matrices can be made in blocks, by \emph{divide-and-conquere} approach one can generalize the Strassen’s algorithm to bigger matrices. As a consequence, matrix multiplication can be performed by using on the order of $n^{log_2(7)} \approx n^{2.81}$ arithmetic operations, in contrary to the standard algorithm which uses the order of $n^3$. The natural question is: what is the smallest possible exponent?

In \cite{burgisser_clausen_shokrollahi} authors proved that looking for the answer, we do not need to worry about the number of additions. To be more precise --- the exponent of the order of the required arithmetic operations equals the exponent of the order of the required multiplications. 
From 1990 until 2010 the smallest known exponent was $2.375477$ \cite{CW87}, given by the Coppersmith–Winograd algorithm.
The state of the art is $2.3728639$ \cite{G14}.
The famous conjecture in algebraic complexity theory states that the number is exactly 2 \cite[Subsect. 3.9]{Landsberg_Complexity}. Roughly speaking, it says that as matrices get large, it becomes as easy to multiply them as to add them. For a more detailed description see \cite{Lan17}.

The matrix multiplication is a bilinear map $\mu_{i,j,k}\colon \kk^{i\times j} \times \kk^{j\times k} \to \kk^{i \times k}$. 
   We can interpret $\mu_{i,j,k}$ as a three-way tensor $\mu_{i,j,k}  \in (\kk^{i \times j})^* \otimes (\kk^{j \times k})^*\otimes \kk^{i \times k}$. 
      The question about the minimal number of multiplications required to calculate the product of two matrices $M, N$, 
  for any $M\in \kk^{i \times j}$ and $N \in \kk^{j \times k}$
is the same as 
``what is the \emph{tensor rank} of $\mu_{i,j,k}$?''.

\begin{example}\label{ex:mu_222}
 The $2\times2$ matrix multiplication tensor is $\mu_{2,2,2} =
 (a_{1,1}\otimes b_{1,1} + a_{1,2}\otimes b_{2,1})\otimes c_{1,1} + (a_{1,1}\otimes b_{1,2} + a_{1,2} \otimes b_{2,2})\otimes c_{1,2} + ( a_{2,1}\otimes b_{1,2} + a_{2,2}\otimes b_{2,1})\otimes c_{2,1} + (a_{2,1}\otimes b_{1,2} + a_{2,2}\otimes b_{2,2})\otimes c_{2,2}$.
 
 Using Theorem \ref{t:strassen_multiplication} we can rewrite it as $\mu_{2,2,2} = 
 (a_{1,1} + a_{2,2} )\otimes (b_{1,1} + b_{2,2} )\otimes (c_{1,1}+ c_{2,2}) 
 + (a_{2,1} + a_{2,2} )\otimes b_{1,1} \otimes (c_{1,2} - c_{2,2}) +
a_{1,1} \otimes  (b_{1,2} - b_{2,2}) \otimes (c_{2,1}+ c_{2,2}) +
a_{2,2} \otimes (-b_{1,1} + b_{2,1}) \otimes (c_{1,1}+ c_{1,2}) +
(a_{1,1} + a_{1,2} )\otimes b_{2,2} \otimes (-c_{1,1}+ c_{1,2}) +
(-a_{1,1} + a_{2,1} )\otimes (b_{1,1} + b_{1,2} )\otimes  c_{2,2} +
(a_{1,2} - a_{2,2} )\otimes (b_{2,1} + b_{2,2} )\otimes c_{1,1}
 $ and see that the rank is at most 7.
\end{example}

\subsection{Strassen additivity problem}
   
One of our main interest is the \emph{additivity} of the tensor rank. Given arbitrary four matrices 
  $M'\in \kk^{i' \times j'}$, $N' \in \kk^{j' \times k'}$, $M''\in \kk^{i'' \times j''}$, $N'' \subseteq \kk^{j'' \times k''}$,
  suppose we want to calculate both products $M' N'$ and $M''N''$ simultaneously.
What is the minimal number of multiplications needed to obtain the result? 
Is it equal to the sum of the ranks $R(\mu_{i',j',k'}) + R(\mu_{i'',j'',k''})$?
More generally, the same question can be asked for  arbitrary tensors. 
A positive answer was widely known as Strassen's Conjecture \cite[p.~194, \S4, Vermutung~3]{strassen_vermeidung_von_divisionen}, \cite[Sect.~5.7]{landsberg_tensorbook}.

\begin{definition}\label{def_rank_additivity}
Assume $A = A' \oplus A''$, $B = B' \oplus B''$, and $C = C' \oplus C''$, where all $\fromto{A}{C''}$ are finite dimensional vector spaces over a field $\kk$.
Pick $p' \in A' \otimes B' \otimes C'$ and $p'' \in A'' \otimes B'' \otimes C''$ 
  and let $p= p' + p''$, which we will write as $p= p'\oplus p''$. We say that the pair $p',p''$ has a \emph{rank additivity property} if the following equality holds
\begin{equation}\label{equ_additivity}
  R(p) = R(p') + R(p'').
\end{equation}
\end{definition}

\begin{problem}[Strassen's additivity problem]\label{prob_strassen}
Given $p',p''$ as in the definition \ref{def_rank_additivity} decide if they poses rank additivity property.
\end{problem}

\begin{thm}[Strassen's additivity does not always hold, \cite{shitov_counter_example_to_Strassen}] \label{t:shitov_counter_example}
There exist
$p' \in A' \otimes B' \otimes C'$ and $p'' \in A'' \otimes B'' \otimes C''$, where $A'=A''=...=C'=C''=\CC^n$ and $n\geq 450$ such that 
\begin{equation}
  R(p' \oplus p'') < R(p') + R(p'')?
\end{equation}
\end{thm}

Shitov did not gave a constructive proof, so there is still work needed to find an explicit example of a pair without rank additivity property.
It is known that if one of the factor vector spaces is small, then the additivity of the tensor rank holds. 
    
\begin{thm}[\cite{jaja_takche_Strassen_conjecture}]\label{t:jaja_2_0_hook}
Using notation from Definition~\ref{def_rank_additivity}, if one of the vector space $A'$, $A''$, $B'$, $B''$, $C'$, $C''$ over an arbitrary field $\kk$ has dimension bounded by 2, then
\begin{equation*}
  R(p' \oplus p'') = R(p') + R(p'').
\end{equation*} 
\end{thm}
 
 See \cite{jaja_takche_Strassen_conjecture} for the original proof     and Section~\ref{sec_hook_shaped_spaces} for a discussion of more recent approaches.

 In the article \cite{BPR} authors address several cases of Problem~\ref{prob_strassen} and its gen\-er\-al\-isa\-tions.
The following theorem summarizes their main results regarding tensor rank.

\begin{thm}[{\cite[Thm. 1.2]{BPR}}]\label{thm_additivity_rank_intro_jarek}
Using notation as in Theorem \ref{thm_additivity_rank_intro_filip}.
  If at least one of the following conditions holds, 
     then the additivity of the rank holds for $p$, that is $R(p) = R(p') + R(p'')$:
  \begin{enumerate}
   \item $R(p'')\le \bfa'' +2$ and $p''$ is not contained in $\tilde{A''} \otimes B'' \otimes C''$
          for any linear subspace $\tilde{A''} \subsetneqq A''$
          (this part of the statement is valid for any base field $\kk$).
          \label{it:1_thm_additivity_rank_intro_jarek}
   \item $\kk=\RR$ (real numbers) or $\kk$ 
            is an algebraically closed field of characteristic $\ne 2$ 
            and $R(p'')\le 6$, 
            \label{it:2_thm_additivity_rank_intro_jarek}
\item $\kk=\CC$ or $\kk=\RR$ (complex or real numbers) and $p''\in A'' \otimes \kk^{3} \otimes \kk^{3}$
   \label{it:3_thm_additivity_rank_intro_jarek}
\end{enumerate}
  Analogous statements hold if we exchange the roles of $A$, $B$, $C$ and/or of ${}'$ and ${}''$. 
\end{thm}


In Sections \ref{s:ranks_and_slices} and \ref{s:dir_sums} we introduce notions and recall tools needed to prove our main results in Section \ref{s:rank_one_matrices_and_additive_rank}.

\section{Ranks and slices}\label{s:ranks_and_slices}

This section reviews the notions of rank, slices and conciseness. Readers that are familiar to these concepts may easily skip this section.
The main things to remember from here are Notation~\ref{notation_V_Seg} and 
  Proposition~\ref{prop_bound_on_rank_for_non_concise_decompositions_for_vector_spaces}.

Throughout this article, similarly as in Notation~\ref{n:ABCabc}, let $\fromto{A_1, A_2}{A_d}$, $A$, $B$, $C$, $V$ and $W$ be finite dimensional vector spaces over 
a field $\kk$. By the bold lowercase letters $\bfa_1, \bfa_2,...,\bfa_d, \bfa, \bfb, \bfc,\mathbf{v},\mathbf{w}$ we denote their dimensions.
If $P$ is a subset of $V$,  we denote by $\langle P \rangle$ its linear span. We will use the same notation, i.e. $\langle P \rangle$ for a projective span if $P$ is a subset of classes of points from a projective space $\PP^N$.
If $P=\setfromto{p_1, p_2}{p_r}$ is a finite subset,  we will write
  $\langle\fromto{p_1, p_2}{p_r}\rangle$ rather than  $\langle\setfromto{p_1, p_2}{p_r}\rangle$ to simplify notation. 

  \subsection{Geometry of secants}
To state the generalization of Definition \ref{d:3way_tensor_rank}, we need to observe that the set of simple tensors is naturally isomorphic to the Cartesian product of projective spaces.
The image of the embedding in the tensor space is called the \emph{Segre variety}.

\begin{defin}\label{d:Segre_variety}
 For $A_1,A_2,\ldots A_d$ vector spaces over $ \kk$, the \emph{Segre variety} is defined as the image of the map, called \emph{Segre embedding}:
\begin{align*}
 Seg: \PP A_1 \times \PP A_2 \times \dots \times \PP A_d&\to \PP(A_1 \otimes A_2 \otimes \dots \otimes A_d) \\
 ([a_1],[a_2],\ldots,[a_d]) &\mapsto [a_1 \otimes a_2 \otimes \dots \otimes a_d].
\end{align*}
If there is no risk of confusion we will denote the image by 

\[
  \Seg=\Seg_{A_1,A_2,\ldots A_d} := \PP A_1 \times \PP A_2 \times \dots \times \PP A_d \subset \PP(A_1 \otimes A_2 \otimes \dots \otimes A_d).
\]
\end{defin}

%

 \begin{defin}\label{def_r_linear_space}
   For a projective variety $\Seg_{A_1,A_2,\ldots A_d} \subseteq \PP^{N}$ and $\PP^k \simeq \PP(W)$  a projective linear subspace of $\PP^N$,
      define $R_{\Seg_{A_1,A_2,\ldots A_d}}(W)$ and $R_{\Seg_{A_1,A_2,\ldots A_d}}(\PP W)$, \emph{the rank} of $W$ and \emph{the rank} of  $\PP(W)$ with respect to ${\Seg_{A_1,A_2,\ldots A_d}}$,
      to be the minimal number $r$    
      such that there exist $r$ classes of points $\{\fromto{[s_1], [s_2]}{[s_r]}\} \subset \Seg_{A_1,A_2,\ldots A_d}$ 
     with $\PP(W)$ contained in $\langle\fromto{s_1, s_2}{s_r}\rangle$.
     
To ease the notation for $[p] \in \PP W$,
we say that rank of a point $p$ is the rank of a projective class of point $[p] \in \PP W$. 
This number will be denoted by  
$R_{\Seg_{A_1,A_2,\ldots A_d}}(p):= R_{\Seg_{A_1,A_2,\ldots A_d}}(\langle p\rangle)$. 
Compare to Definition \ref{d:3way_tensor_rank}.

   We will drop $\Seg_{A_1,A_2,\ldots A_d}$ from the subscript, if the variety we work with is clear from the context.
\end{defin}

  
In the setting of Definition~\ref{def_r_linear_space}, if $X = \Seg_{A_1,A_2,\dots,A_d}$ and $W \subseteq A_1 \otimes A_2 \otimes \dots \otimes A_d$ and $d=1$, then $R(W)=R(\PP W)= \dim W$.
If $d=2$ and $W = \langle p \rangle$ is $1$-dimensional, then $R(W)$ is the rank of $p$ viewed as a linear map $A_1^*\to A_2$.
If $d=3$ and $W = \langle p \rangle$ is $1$-dimensional, then $R(W)$ is equal to $R(p)$ in the sense of Definition~\ref{def_r_linear_space}.

More generally, for an arbitrary $d$, one can relate the rank $R(p)$ of $d$-way tensors with the rank $R(W)$ of certain linear subspaces in the space of $(d-1)$-way tensors. 
This relation is based on the \emph{slice technique}, which we are going to review in Section~\ref{section_slices}.

A central task in many problems is to test tensor membership in a given set (e.g., if a tensor has rank $r$). Some of these sets are defined as the zero sets of collections of polynomials, i.e. as algebraic varieties. However in general, the set of tensors of rank at most $r$ is neither open nor closed.
One of the very few exceptions is the case of matrices, that is tensors in $A\otimes B$.

\subsection{Variety of simple tensors}\label{ss:variety_of_simple_tensors}

We will intersect linear subspaces of the tensor space with the Segre variety.
Using the language of algebraic geometry, such intersection may have a non-trivial scheme structure.
In Section \ref{s:rank_one_matrices_and_additive_rank} 
we just ignore the scheme structure and all  our intersections are set theoretic.
To avoid ambiguity of notation,
we write $\reduced{(\cdot)}$ to underline this issue, while
the reader not originating from algebraic geometry should ignore the symbol $\reduced{(\cdot)}$.

\begin{notation}\label{notation_V_Seg}
  Given a linear subspace of a tensor space, $V \subseteq A_1 \otimes A_2 \otimes \dotsb \otimes A_d$,
    we denote:
\[
  V_{\Seg}:=\reduced{(\PP V  \cap \Seg_{A_1,A_2,\dots,A_d})}. 
\]
Thus, $V_{\Seg}$ is (up to projectivization) 
   the set of rank one tensors in $V$.
\end{notation}

In this setting, we have the following trivial rephrasing of the definition of rank:
\begin{prop}[{\cite[Prop. 2.3.]{BPR}}]\label{prop_can_choose_decomposition_containing_simple_tensors_from_W}
   Suppose $W \subseteq A_1 \otimes A_2 \otimes \dotsb \otimes A_d$ is a linear subspace.
   Then $R(W)$ is equal to the minimal number $r$ such that there exists a linear subspace $V \subseteq A_1 \otimes A_2 \otimes \dotsb \otimes A_d$ of dimension $r$
     with $W \subseteq V$ and $\PP V $ is linearly spanned by 
     $V_{\Seg}$.
   In particular,
   \begin{enumerate}
      \item $R(W) = \dim W$ if and only if  
            \[
               \PP  W = \langle W_{\Seg} \rangle.
            \]  
      \item Let $U$ be the linear subspace such that $\PP U :=\langle W_{\Seg} \rangle$. 
            Then $\dim U$ tensors from $W$ can be used in the minimal decomposition of $W$, 
             that is there exist $\fromto{s_1}{s_{\dim U}}\in W_{\Seg}$ such that $W\subset \langle \fromto{s_1}{s_{R(W)}} \rangle$ and $s_i$ are simple tensors.
   \end{enumerate}
\end{prop}

\subsection{Slice technique and conciseness}\label{section_slices}

We define the notion of conciseness of tensors and we review 
   a standard \emph{slice technique} that replaces the calculation of rank of three way tensors with the calculation of rank of linear spaces of matrices. 

A tensor $p \in A_1 \otimes A_2 \otimes \dotsb \otimes A_d$ determines a linear map $p \colon A_1^* \to A_2 \otimes \dotsb \otimes A_d$. 
If we choose a basis $\{a_1,a_2, \ldots, a_\bfa\}$ of $A_1$ we can write
$$
p=\sum_{i=1}^{\bfa} a_i\ts w_i,
$$
where $w_1,\dots,w_{\bfa}\in W:=p({A_1}^\ast)\subset A_2 \otimes \dotsb \otimes A_d$.

The elements $w_1,\dots,w_{\bfa}\in W$ are called \emph{slices} of $p$.
The point is that $W$ essentially uniquely (up to an action of $GL(A_1)$) determines $p$ (cf. \cite[Cor.~3.6]{landsberg_jabu_ranks_of_tensors}).
Thus, the subspace $W$ captures the geometric information about $p$, in particular its rank and border rank.

\begin{lem}[{\cite[Thm~2.5]{landsberg_jabu_ranks_of_tensors}}]\label{lem_rank_of_space_equal_to_rank_of_tensor}
   Suppose $p \in A_1 \otimes A_2 \otimes \dotsb \otimes A_d$ and $W = p(A_1^*)$ as above. 
   Then $R(p) = R(W)$ 
.
\end{lem}

Clearly, we may also replace  $A_1$ with any of the $A_i$ to define slices as images $p(A_i^*)$ and obtain the analogue of the lemma. 
The technique of slicing a tensor is classical. The relations between the rank of a tensor and the rank of the space defined by its slices were found by Terracini \cite{Ter1915}.

\subsection{Independence of the rank of the ambient space}

As defined above, the notions of rank and border rank of a vector subspace  $W \subseteq A_1\otimes A_2 \otimes\dotsb \otimes A_d$,
   or of a tensor $p \in A_1 \otimes A_2 \otimes\dotsb \otimes A_d$, might seem to  depend on the ambient spaces $A_i$.
However, it is well known, that the rank is actually independent of the choice of the vector spaces.
Also a stronger fact about the rank is true.
Suppose $p \in A_1' \otimes A_2' \otimes \dotsb \otimes A_d'$ for some linear subspaces $A_i'\subset A_i$.
  Any minimal expression $p\subseteq \langle \fromto{s_1}{s_{R(p)}}\rangle$, 
  for simple tensors $s_i$, must be contained in  $A_1' \otimes \dotsb \otimes A_d'$.
In \cite{BPR}, authors show that the difference in the length of the decompositions 
  must be at least the difference of the respective dimensions.
We stress that the lemma below does not depend on the base field, in particular, it does not require algebraic closedness.

\begin{lem}[{\cite[Lem. 2.8]{BPR}}]\label{lem_bound_on_rank_for_non_concise_decompositions}
   Suppose that $p \in A_1' \otimes A_2 \otimes A_3 \otimes \dotsb \otimes A_d$, for a linear subspace $A_1'\subset A_1$, 
      and that we have an expression $p \in \langle \fromto{s_1}{s_{r}}\rangle$, where $s_i = a_i^1 \otimes a_i^2 \otimes \dots \otimes a_i^d$ are simple tensors.
  Then:
  \[
     R(p) + \dim \langle \fromto{a_1^1}{a_r^1}\rangle - \dim A' \leq r.
  \]
\end{lem}

The analogue of Lemma~\ref{lem_bound_on_rank_for_non_concise_decompositions}
  for higher dimensional subspaces of the tensor space is also true.

\begin{prop}[{\cite[Prop. 2.10.]{BPR}}]\label{prop_bound_on_rank_for_non_concise_decompositions_for_vector_spaces}
   Suppose $W \subset A_2' \otimes \dotsb \otimes A_d'$ for some linear subspaces $A_2'\subset A_2$,..., $A_d' \subset A_d$. 
\begin{enumerate}
   \item \label{item_rank_can_be_measured_anywhere}
      The number $R(W)$ 
      measured as the rank 
      of $W$ in $A_2' \otimes \dotsb \otimes A_d'$
        is equal to its rank 
        calculated in $A_2 \otimes \dotsb \otimes A_d$ 
        .
   \item \label{item_decompositions_in_larger_spaces}
      Moreover, if we have an expression $W \subset \langle \fromto{s_1}{s_{r}}\rangle$, 
        where $s_i = a_i^2 \otimes a_i^3 \otimes \dots \otimes a_i^d$ are simple tensors,
        then:
  \[
    R(W) + \dim \langle \fromto{a_1^2}{a_r^2}\rangle - \dim A_2' \leq r
  \]
\end{enumerate}

\end{prop}


We conclude this section by recalling the following definition.
\begin{defin}
   Let $p \in A_1 \otimes A_2 \otimes \dotsb \otimes A_d$ be a tensor or let $W \subset A_1 \otimes A_2 \otimes \dotsb \otimes A_d$ be a linear subspace. 
   We say that $p$ or $W$ is \emph{$A_1$-concise} if for all linear subspaces $V \subset A_1$, if $p \in V \otimes A_2 \otimes \dotsb \otimes A_d$
      (respectively,  $W \subset V \otimes A_2 \otimes \dotsb \otimes A_d$), then $V = A_1$.
   Analogously, we define $A_i$-concise tensors and spaces for $i =\fromto{2}{d}$. 
   We say $p$ or $W$ is \emph{concise} if it is $A_i$-concise for all $i\in \setfromto{1}{n}$.
\end{defin}

\begin{rem}
Notice, that $p \in A_1 \otimes A_2 \otimes \dotsb \otimes A_d$ is $A_1$-concise if and only if $p\colon A_1^* \to A_2 \otimes \dotsb \otimes A_d$ is injective.
In particular, from injectivity and Lemma \ref{lem_rank_of_space_equal_to_rank_of_tensor} follows that rank of a $A_1$-concise tensor is greater or equal than the dimension of~$A_1$.
\end{rem}

\section{Direct sum tensors and spaces of matrices}\label{s:dir_sums}
In this section we introduce the notation coming from 
\cite{BPR} which will be adopted throughout Sections \ref{s:dir_sums}
and \ref{s:rank_one_matrices_and_additive_rank}.
For simplicity we restrict the presentation to the case of tensors in $A\otimes B \otimes C$ or linear subspaces of $B \otimes C$.

\begin{notation}\label{notation}
Let $A = A' \oplus A''$,  $B = B' \oplus B''$, $C = C' \oplus C''$ be vector spaces over $\kk$ of dimensions $\bfa=\dim A = \bfa'+\bfa''$, $\bfb=\dim B =\bfb'+\bfb''$ and $\bfc=\dim C =\bfc'+\bfc''$.

For the purpose of illustration, we will interpret the two-way tensors in $B \otimes C$ 
as matrices in $\kk^{\bfb\times \bfc}$. 
This requires choosing bases of $B$ and $C$, but (whenever possible) we will refrain from naming the bases explicitly.
We will refer to an element of the space of matrices 
$\kk^{\bfb\times \bfc} \simeq B\ts C$ as a $(\bfb'+\bfb'',\bfc'+\bfc'')$ \emph{partitioned matrix}.
Every matrix $w\in\kk^{\bfb\times \bfc} $ is a block matrix with four blocks of size
$\bfb'\times \bfc'$, $\bfb'\times \bfc''$, 
$\bfb''\times \bfc'$ and $\bfb''\times \bfc''$ respectively.
\end{notation}

\begin{notation}\label{notation_2}
As in Section~\ref{section_slices}, a tensor $p\in A\otimes B\otimes C$ is a linear map $p:A^\ast\to B\otimes C$;
   we denote the image of $A^\ast$ in the space of matrices $B\otimes C$  by $W:=p(A^\ast)$ .
   Similarly, if $p = p' \oplus p'' \in (A' \oplus A'') \otimes (B' \oplus B'') \otimes (C' \oplus C'')$ is such that  $p'\in A'\otimes  B'\otimes C'$ and $p''\in A''\otimes  B''\otimes C''$, we set $W':=p'({A'}^\ast)\subset B'\otimes C'$ and $W'':=p''({A''}^\ast)\subset B''\otimes C''$. 
In such situation, we will say that $p = p' \oplus p''$ is a \emph{direct sum tensor}.

We have the following direct sum decomposition:
\[
  W=W'\oplus W''\subset (B'\otimes C')\oplus (B''\otimes C'')
\]
and an induced matrix partition of type $(\bfb'+\bfb'',\bfc'+\bfc'')$ on every matrix $w\in W$  such that 
\[
  w=\begin{pmatrix} 
       w'           & \underline{0} \\
      \underline{0} & w''
    \end{pmatrix},
\]
where $w'\in W'$ and $w''\in W''$, and the two $\underline{0}$'s denote zero matrices of size
$\bfb'\times\bfc''$ and $\bfb''\times\bfc'$ respectively.
\end{notation}

\begin{prop}[{\cite[Prop. 3.3.]{BPR}}]\label{prop_translete_to_subspace_language}
Suppose that $p$, $W$, etc.~are as in Notation~\ref{notation_2}.
Then the additivity of the rank holds for $p$, that is~$R(p) = R(p')+R(p'')$, if and only if the additivity of the rank holds for $W$, that is $R(W) = R(W')+ R(W'')$.
\end{prop}

\subsection{Projections and decompositions} \label{ss:proj_and_decomp}

The situation we consider here concerns the direct sums and their minimal decompositions.
We fix $W' \subset B'\otimes C'$ and $W'' \subset B''\otimes C''$ and we choose a minimal decomposition of $W'\oplus W''$ (see Section \ref{ss:variety_of_simple_tensors}), 
that is a linear subspace $V\subset B \otimes C$
  such that $\dim V = R(W'\oplus W'')$, $\PP V=\linspan{V_{\Seg}}$ and $V\supset W'\oplus W''$.
Such linear spaces $W'$, $W''$ and $V$ will be fixed for the rest of Sections~\ref{s:dir_sums} and \ref{s:rank_one_matrices_and_additive_rank}.

In addition to  Notations~\ref{notation_V_Seg}, \ref{notation} and \ref{notation_2} we need the following.

\begin{notation}\label{notation_projection}
Under Notation~\ref{notation},  let $\pi_{C'}$ denote the projection
\[
   \pi_{C'}:C \to C'',
\]
whose kernel is the space $C'$. With slight abuse of notation, we shall  denote by $\pi_{C'}$ also the following projections
   \[
   \pi_{C'}:B\ts C\to B\ts C'', \text{ or } \pi_{C'}:A\otimes B\ts C\to A \otimes B\ts C'',
\]
with kernels, respectively, $B\otimes C'$ and $A\otimes B \otimes C'$.
The target of the projection is regarded as a subspace of $C$, $B\otimes C$, or $A\otimes B \otimes C$, so that it is possible to compose such projections, for instance:
    \[
        \pi_{C'} \pi_{B''} \colon B\otimes C \to B'\otimes C'', \text{ and } \pi_{C'} \pi_{B''} \colon A \otimes  B\otimes C \to A\otimes B'\otimes C''.
    \]
We also let $E'\subset B' $ (resp. $E''\subset B''$) be the minimal vector subspace such that 
   $\pi_{C'}(V)$ (resp. $\pi_{C''}(V)$) is contained in $(E'\oplus B'')\ts C''$ (resp. $(B'\oplus E'')\ts C'$).

By swapping the roles of $B$ and $C$, we  define $F'\subset C'$ and $F''\subset C''$ analogously. By the lowercase letters $\bfe',\bfe'',\bff',\bff''$ we denote the dimensions of the subspaces $E',E'',F',F''$.
\end{notation}

\begin{figure}
\centering
\includegraphics[scale=0.06]{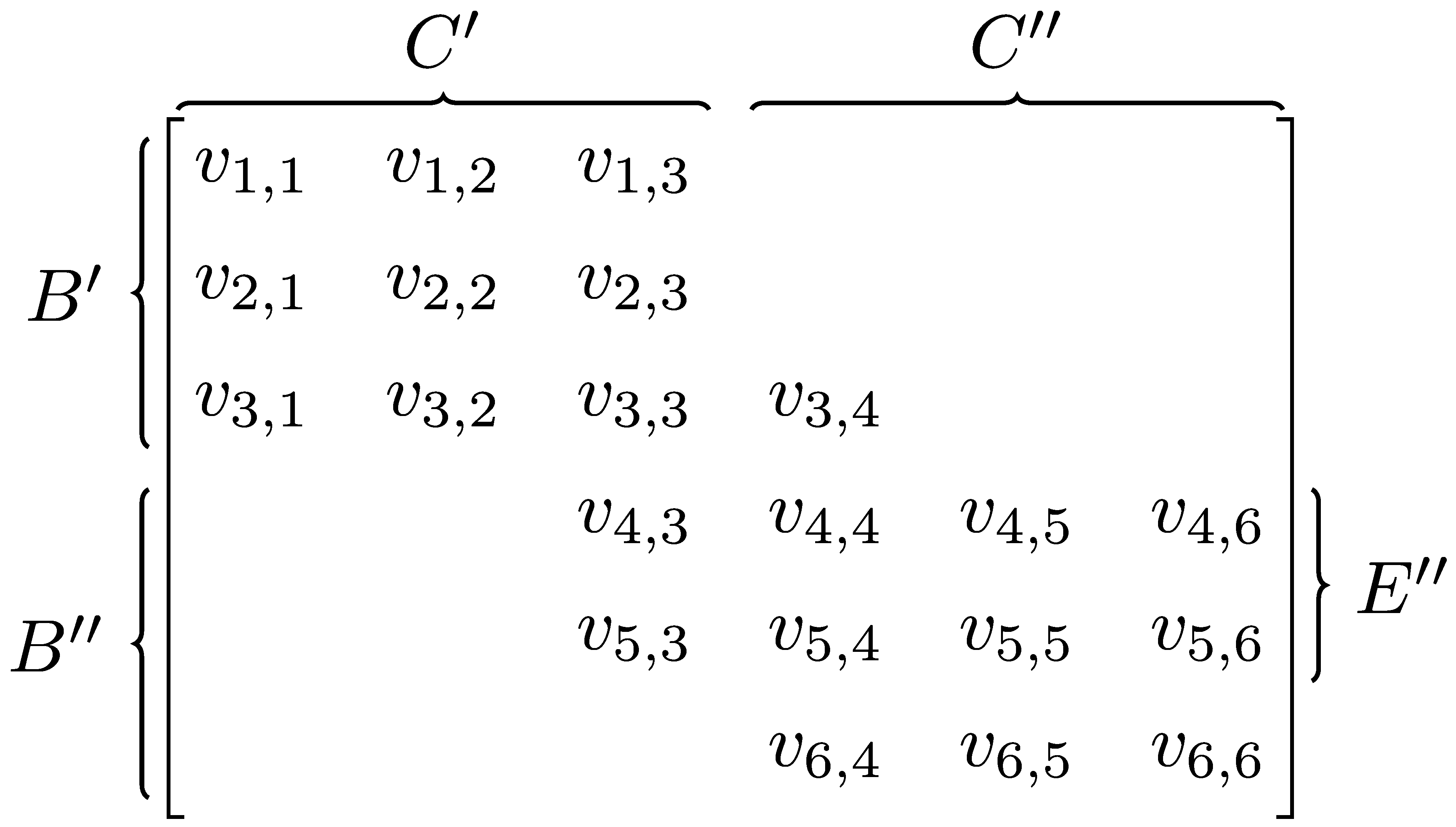}
\caption{A minimal decomposition of $W'\oplus W''$, 
that is a linear subspace $V\subset B \otimes C$
  such that $\dim V = R(W'\oplus W'')$, $\PP V=\linspan{V_{\Seg}}$ and $V\supset W'\oplus W''$.
  We denote by $E''\subset B'' $ the minimal vector subspace such that $\pi_{C''}(V) \subset B \otimes C'$ is contained in $(B'\oplus E'')\otimes C'$. 
In the presented case $(\bfb',\bfb'',\bfc',\bfc'')=(3,3,3,3)$ (we use Notation~\ref{notation}). 
} \label{fig_projections}
\end{figure}

If the differences $R(W') - \dim W'$ and $R(W'') - \dim W''$ (which we will informally call the \emph{gaps}) are large, then the spaces $E',E'',F',F''$ could be large too, in particular they can coincide with $B', B'', C', C''$ respectively. 
In fact, these spaces measure ``how far'' a minimal decomposition $V$ of a direct sum $W=W' \oplus W''$ is from being a direct sum of decompositions of $W'$ and $W''$. 

In particular, 
 if $E'' = \set{0}$ or if both $E''$ and $F''$ are sufficiently small, then $R(W) = R(W')+ R(W'')$.

\begin{lem}[{\cite[Lem. 3.5]{BPR}}]\label{lemma_bound_r'_e'_R_w'}
In Notation~\ref{notation_projection} as above, with $W=W'\oplus W'' \subset B\otimes C$, 
the following inequalities hold.
\begin{align*}
R(W') + \bfe'' & \le R(W)-\dim W'',&
R(W'')+ \bfe'  & \le R(W)-\dim W', \\
R(W') + \bff'' & \le R(W)-\dim W'',&
R(W'')+ \bff'  & \le R(W)-\dim W'.
\end{align*}
\end{lem}

Rephrasing the inequalities of Lemma~\ref{lemma_bound_r'_e'_R_w'}, we obtain the following.
\begin{cor}[{\cite[Cor. 3.6.]{BPR}}]\label{cor_bounds_on_es_and_fs}
   If $R(W) < R(W')+R(W'')$, then 
   \begin{align*}
       \bfe' &< R(W')  - \dim W', &
       \bff' &< R(W')  - \dim W', \\
       \bfe''&< R(W'') - \dim W'',&
       \bff''&< R(W'') - \dim W''.
   \end{align*}
\end{cor}
This immediately recovers a known case of additivity, when the gap is equal to $0$ \cite[Prop. 10.3.3.3]{landsberg_tensorbook}.
Moreover, it implies that if one of the gaps is equal to $1$ 
   (say  $R(W')=\dim W'+1$),
   then the additivity holds. Even more is true. It is sufficient to assume that only one of $E'$ or $F'$ is zero.

  \begin{prop}[{\cite[Lem. 3.7, Prop. 4.4.]{BPR}}]\label{prop_SAC_if_E'=0}\label{lem_rank_at_least_2_more_than_dimension}
With Notation~\ref{notation_projection}, 
if one among $E',E'',F',F''$ is zero, then $R(W)=R(W')+R(W'')$.
   In particular, if $R(W') \le \dim W' +1$, then the additivity holds.
\end{prop}

%

We show, as a consequence of Corollary~\ref{cor_bounds_on_es_and_fs}, that
  if one of the gaps is at most two, 
  then the additivity of the rank holds, see Theorem~\ref{thm_additivity_rank_intro_jarek}~\ref{it:1_thm_additivity_rank_intro_jarek}.
In Corollary~\ref{cor_1_2_hook_shaped} 
   we prove a further generalization based on the results in the following subsection.

\subsection{``Hook''-shaped spaces and the substitution method}\label{sec_hook_shaped_spaces}

It is known since \cite{jaja_takche_Strassen_conjecture}, that the additivity of the tensor rank holds for tensors with one of the factors of dimension $2$ (Theorem~\ref{t:jaja_2_0_hook}). Namely, using Notation~\ref{notation} and \ref{notation_2}, 
  if $\bfa' \le 2$ then $R(p'+p'') = R(p')+R(p'')$.
The same claim over algebraically closed fields is proved in \cite[Prop.~3.2.12]{rupniewski_mgr}
using \emph{substitution method} (Propositions~\ref{proposition_for_AFT_method_coordinate_free} and \ref{proposition_for_AFT_method_slice_A}).
%
We generalize it in Corollary~\ref{cor_1_2_hook_shaped}.
To state and prove the generalization in Section \ref{s:rank_one_matrices_and_additive_rank} we introduce (following \cite{BPR}) the notion of ``hook''-shaped spaces. We stress that comparing to \cite[Prop.~3.17]{BPR} our proof works over any base field. 

\begin{defin}\label{def_hook_shaped_space}
   For non-negative integers $e, f$, we say that a linear subspace $W \subset B \otimes C$ 
      is \emph{$(e,f)$-hook shaped}, 
      if $W\subset \kk^e\otimes C + B \otimes \kk^f$ 
      for some choices of linear subspaces $\kk^e\subset B$ and $\kk^f\subset C$.
\end{defin}

The name ``hook shaped'' space comes from the fact that under an appropriate choice of basis, the only nonzero coordinates form a shape of a hook 
$\ulcorner$ situated in the upper left corner of the matrix, see Example~\ref{ex_hook}.
The integers $(e,f)$ specify how wide the edges of the hook are.
A similar name also appears in the context of Young diagrams, see for instance \cite[Def.~2.3]{berele_regev_hook_Young_diagrams_with_applications}.

\begin{example}\label{ex_hook}
A $(1,2)$-hook shaped subspace of $\kk^4 \otimes \kk^4$ has only the following possibly nonzero entries in some coordinates:
$$\begin{bmatrix}
  *   & * & * & *\\
  * & * & 0 & 0\\
  * & *  & 0 & 0\\
  * & *  & 0 & 0\\
 \end{bmatrix}.
$$
\end{example}

The following 
observation is presented in \cite[ Prop.~3.1]{landsberg_michalek_abelian_tensors} and in \cite[Lem.~B.1]{alexeev_forbes_tsimerman_Tensor_rank_some_lower_and_upper_bounds}.
Here, we have phrased it in a coordinate free way.

\begin{prop}[{\cite[Prop. 3.10]{BPR}}] \label{proposition_for_AFT_method_coordinate_free} 
   Let $p \in A\otimes B \otimes C$, $R(p) = r>0$, and pick $\alpha \in A^*$ such that $p(\alpha) \in B\otimes C$ is nonzero.
   Consider two hyperplanes in $A$: the linear hyperplane $\alpha^{\perp}= (\alpha =0)$
     and the affine hyperplane $(\alpha=1)$.
   For any $a\in (\alpha=1)$, denote
   \[
     \tilde{p}_{a}:= p - a\otimes p(\alpha) \in \alpha^{\perp} \otimes B \otimes C.
   \]
   Then:
 \begin{enumerate}
  \item \label{item_AFT_exists_choice_droping_rank}
        there exists a choice of $a\in (\alpha=1)$ such that $R(\tilde{p}_{a}) \leq r-1$,
  \item \label{item_AFT_any_choice_drops_rank_at_most_1}
        if in addition $R(p(\alpha))=1$, then for any choice of $a \in (\alpha=1)$ 
        we have $R(\tilde{p}_{a}) \geq r-1$.
 \end{enumerate}
\end{prop}
See \cite[Prop.~3.1]{landsberg_michalek_abelian_tensors} for the proof (note the statement there is over the complex numbers only, but the proof is base field independent) or, alternatively,
   using Lemma~\ref{lem_rank_of_space_equal_to_rank_of_tensor} translate it into 
   the following  statement on linear spaces of tensors:

\begin{prop}[{\cite[Prop. 3.11]{BPR}}]\label{proposition_for_AFT_method_slice_A} 
 Suppose $W \subset  B \otimes C$ is a linear subspace, $R(W) = r$. 
 Assume $w \in W$ is a nonzero element.
 Then:
 \begin{enumerate}
  \item \label{it:proposition_for_AFT_method_slice_A_1} there exists a choice of a complementary subspace $\widetilde{W}\subset W$
     such that $\widetilde{W} \oplus \linspan{w} = W$ and $R(\widetilde{W}) \leq r-1$, and
  \item if in addition $R(w)=1$, then for any choice of the complementary subspace $\widetilde{W} \oplus \linspan{w} = W$
     we have $R(\widetilde{W}) \geq r-1$.
 \end{enumerate}
\end{prop}

Proposition~\ref{proposition_for_AFT_method_slice_A} and the following Lemma~\ref{lem_dim_2_have_rank_one_matrix}  were crucial in the original proof that the additivity of the rank holds for vector spaces, one of which is $(1,2)$-hook shaped (provided that the base field is algebraically closed). It is presented in \cite[Subsect.~3.2]{BPR}.

After introducing repletion and digestion with respect to a distinguished matrix (\S \ref{ss:replete_pairs} and \S~\ref{ss:digestion}), we present Corollaries~\ref{cor_can_replete_and_digest} and \ref{c:removing_rank_1_slice}, which are 
 stronger versions of Proposition~\ref{proposition_for_AFT_method_slice_A}\ref{it:proposition_for_AFT_method_slice_A_1}. In particular they implies what follows.


\begin{prop}\label{proposition_for_AFT_method_slice_A_improved} 
 Suppose $W = W' \oplus W'' \subset  B \otimes C$ is a linear subspace and  $w \in W'$ is such that  $R(w)=1$.
 Then there exists a choice of a complementary subspace $\widetilde{W'}\subset W'$
     such that 
    \[ \widetilde{W'} \oplus \linspan{w} \oplus W''= W \text{ and } R(\widetilde{W'}\oplus W'') = R(W)-1.\]
  
\end{prop}


This approach also simplifies the original proof of additivity of the rank holds for vector spaces, one of which is $(1,2)$-hook shaped 
and let us to prove its generalization to arbitrary fields, i.e. Corollary~\ref{cor_1_2_hook_shaped} (and Theorem~\ref{thm_additivity_rank_intro_filip}\ref{it:4_thm_additivity_rank_intro_filip} from \fullref{ch:introduction}).

The proof of the following lemma is a dimension count, see also \cite[Prop.~3.2.11]{rupniewski_mgr}.

\begin{lem}[{\cite{BPR}, Lemma~3.16}]\label{lem_dim_2_have_rank_one_matrix}
    Suppose $\kk$ is algebraically closed (of any characteristic) 
       and $0 \neq p\in A\otimes B \otimes \kk^2$ is concise and $\dim A \geq \dim B$.
    Then, 
there exists a rank one matrix in $p(A^*)\subset B\otimes \kk^2$.
\end{lem}

Our proof of Lemma~\ref{lem_dim_2_have_rank_one_matrix} does not work for non algebraically closed fields, since we rely on \cite[Thm~I.7.2]{hartshorne}.
In this article, we use Lemma~\ref{lem_dim_2_have_rank_one_matrix} and the generalization of Proposition~\ref{proposition_for_AFT_method_slice_A} 
(Corollary~\ref{cor_can_replete_and_digest}), 
to prove that rank additivity holds for a certain small dimensional spaces, see Corollary~\ref{cor_additivity_C_443_443} (and Theorem~\ref{thm_additivity_rank_intro_filip}\ref{it:5_thm_additivity_rank_intro_filip}). 

 
\section{Rank one matrices and additivity of the tensor rank}\label{s:rank_one_matrices_and_additive_rank}

If in the linear space $W'$ or $W''$ we find a rank one matrix, then we have a good starting point for an attempt 
   to prove the additivity of the rank.
Throughout this section we will change this observation to a formal statement and prove, that if there is a rank one matrix in the linear spaces,
   then either the additivity holds or there exists a ``smaller'' example of failure of the additivity. 
In Section~\ref{sect_proofs_of_main_results_on_rank} we exploit several versions of this claim 
   in order to prove Theorem~\ref{thm_additivity_rank_intro_filip}.

Throughout this section we follow Notations~\ref{notation_V_Seg} (denoting the rank one elements in a vector space by 
                                                                             the subscript $\cdot_{Seg}$),
    \ref{notation} (introducing the vector spaces $\fromto{A, A'}{C''}$ and their dimensions $\fromto{\bfa,\bfa'}{\bfc''}$),
       and also~\ref{notation_projection} (which explains the conventions for projections $\fromto{\pi_{A'}, \pi_{A''}}{\pi_{C''}}$ 
       and vector spaces $E', E'', F', F''$, which measure how much the decomposition $V$ of $W$ sticks out 
       from the direct sum $B'\otimes C' \oplus B''\otimes C''$). In this chapter, the letter $V$ will denote the decomposition of a subspace $W$, as defined at the beginning of Subsection \ref{ss:proj_and_decomp}.
       We will also frequently use Notation~\ref{notation_2} and Proposition~\ref{prop_translete_to_subspace_language}. Together they define a direct sum tensor $p=p'\oplus p''$ and let us translate the problem of additivity of rank for tensors to the additivity of rank for the corresponding vector spaces $W, W', W''$.

\subsection{Combinatorial splitting of the decomposition}

We carefully analyze 
   the structure of the rank one matrices in $V$.
We will distinguish seven types of such matrices.

\begin{lem}[{\cite[Lem. 4.1.]{BPR}}]\label{lemma_reduction}
Every element of $V_{\Seg}\subset \PP(B\ts C)$ lies in the projectivization of one of the following subspaces
  of $B\ts C$:
\begin{enumerate}
\item \label{item_prime_bis}
$B'\otimes C'$, $B''\otimes C''$,  \hfill (\Prime, \Bis)
\item  \label{item_vertical_and_horizontal}
$E'\otimes (C'\oplus F'')$,
$E''\otimes (F'\oplus C'')$,\hfill (\HL, \HR)\\
$(B'\oplus E'')\otimes F'$,
$(E'\oplus B'')\otimes F''$,\hfill (\VL, \VR) \item \label{item_mixed}
$(E'\oplus E'')\otimes (F'\oplus F'')$. \hfill (\Mix)
\end{enumerate}
\end{lem}
The spaces in \ref{item_prime_bis} are purely contained in the original direct summands, hence, in some sense, they are the easiest to deal with (we will show how to ``get rid'' of them and construct a smaller example justifying a potential lack of additivity).\footnote{The word \Bis{} comes from the Polish way of pronouncing the ${}''$ symbol.}
The spaces in \ref{item_vertical_and_horizontal} 
  stick out of the original summand, but only in one direction, either horizontal (\HL, \HR), or vertical (\VL, \VR)\footnote{Here, the letters ``H, V, L, R'' stand for 
  ``horizontal, vertical, left, right'' respectively.}.  
The space in \ref{item_mixed} is mixed and it sticks out in all directions. It is the most difficult to deal with and we expect, that the typical counterexamples to the additivity of the rank will have mostly (or only) such mixed matrices in their minimal decomposition. 
The mutual configuration and layout of those spaces in the case $(\bfb',\bfb'',\bfc',\bfc'')=(3,3,3,3)$, $(\bfe',\bfe'',\bff',\bff'')=(1,2,1,1)$ is illustrated in Figure~\ref{fig_layout_of_Prim__Mixed}. We use our usual convention, that bold lower case  letters denote dimensions of the spaces denoted by capital letter. 

 \begin{figure}
\centering
\includegraphics[scale=0.18]{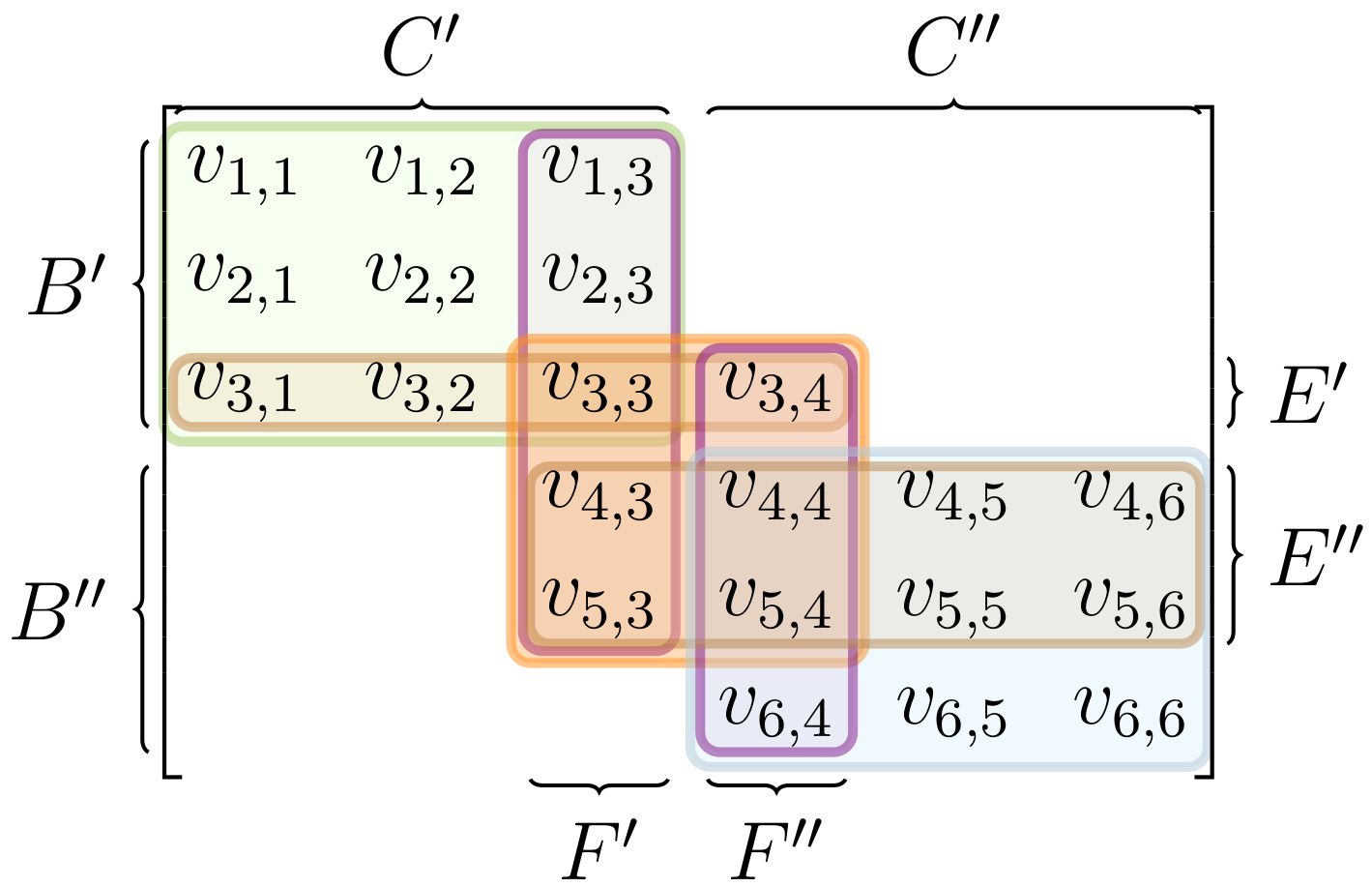}
 \caption{We use Notation~\ref{notation_projection}. 
In the case $(\bfb',\bfb'',\bfc',\bfc'')=(3,3,3,3)$, $(\bfe',\bfe'',\bff',\bff'')=(1,2,1,1)$,
choose a basis of $E'$ and a completion to a basis of $B'$ and, 
similarly, bases for  $(E'', B''), (F', C'),(F'',C'')$. We can represent the elements of $V_{\Seg} \subset B \otimes C$ as matrices in one of the following subspaces:
\Prime{} (corresponding to the top-left green rectangle), 
\Bis{} (bottom-right blue rectangle), 
\VL{} (purple with entries $v_{1,3},v_{2,3},v_{3,3},v_{4,3},v_{5,3}$), 
\VR{} (purple with entries $v_{3,4},v_{4,4},v_{5,4},v_{6,4}$), 
\HL{} (brown with entries $v_{3,1},v_{3,2},v_{3,3},v_{3,4}$),
\HR{} (brown with entries $v_{4,3},v_{4,4},v_{4,5},v_{4,6},v_{5,3},v_{5,4},v_{5,5},v_{5,6}$),
and \Mix{} (middle orange square with entries $v_{3,3},v_{3,4},v_{4,3},v_{4,4},v_{5,3},v_{5,4}$).
} \label{fig_layout_of_Prim__Mixed}
\end{figure}

As in Lemma~\ref{lemma_reduction} every element of $V_{Seg}\subset \PP(B\otimes C)$ 
   lies in one of seven subspaces of $B\otimes C$. 
These subspaces may have nonempty intersection.
We will now explain our convention with respect to choosing a basis of $V$ consisting of elements of $V_{\Seg}$.

Here and throughout the thesis, by $\sqcup$ we denote the disjoint union.

\begin{notation}\label{notation_Prime_etc}
  We choose a basis $\ccB$ of $V$ in such a way that:
  \begin{itemize}
   \item $\ccB$ consist of rank one matrices only,
   \item $\ccB = \Prime \sqcup \Bis \sqcup \HL \sqcup \HR \sqcup \VL \sqcup \VR \sqcup \Mix$, 
           where each of \Prime, \Bis, \HL, \HR, \VL, \VR, and \Mix{} 
           is a finite set of rank one matrices of the respective type as in Lemma~\ref{lemma_reduction} 
           (for instance, $\Prime \subset B'\otimes C'$, $\HL\subset E'\otimes (C'\oplus F'')$, etc.).
   \item $\ccB$ has as many elements of $\Prime$ and $\Bis$ as possible, subject to the first two conditions,
   \item $\ccB$ has as many elements of $\HL$, $\HR$, $\VL$ and $\VR$ as possible, subject to all of the above conditions.
  \end{itemize}
  Let $\prim$ be the number of elements of $\Prime$ (equivalently, $\prim = \dim\linspan{\Prime}$)
     and analogously define $\bis$, $\hl$, $\hr$, $\vl$, $\vr$, and $\mix$.
  The choice of $\ccB$ need not  be unique, but we fix one for the rest of the chapter.
  Instead, the numbers $\prim$, $\bis$, and $\mix$ are uniquely determined by $V$
  (there may be some non-uniqueness in dividing between  $\hl$, $\hr$, $\vl$, $\vr$).
\end{notation}

Thus, to each decomposition we associated a sequence of seven non-negative integers $(\fromto{\prim}{\mix})$.
We now study the inequalities between these integers and exploit them to get theorems about the additivity of the rank.
  
\begin{prop}[{\cite[Prop. 4.3.]{BPR}}]\label{prop_projection_inequalities}
 In Notations~\ref{notation_projection} and \ref{notation_Prime_etc} the following inequalities hold:
 \begin{enumerate}
  \item \label{item_projection_inequality_short_prim}
        $\prim + \hl + \vl 
  + \min\big(\mix,\bfe' \bff'  \big) \geq R(W')$,
  \item \label{item_projection_inequality_short_bis}
        $\bis + \hr  + \vr 
  + \min\big(\mix,\bfe'' \bff'' \big) \geq R(W'')$,
 \end{enumerate}
\end{prop}
\renewcommand{\theenumi}{(\alph{enumi})}

\begin{cor}[{\cite[Cor. 4.5.]{BPR}}]\label{cor_inequalities_HL__Mix}
  Assume that the additivity fails for $W'$ and $W''$, that is, $d=R(W') + R(W'')-R(W'\oplus W'') >0$.
  Then the following inequalities hold:
 \begin{enumerate}
    \item \label{item_mix_at_least_1}
          $\mix \geq d \ge 1$, 
    \item \label{item_mix_and_horizontal_at_least_1_plus_e}
        $\hl + \hr + \mix \geq \bfe' + \bfe'' + d \ge 3$,
    \item \label{item_mix_and_vertical_at_least_1_plus_f}
        $\vl + \vr + \mix \geq \bff' + \bff'' + d \ge 3$.
 \end{enumerate}
\end{cor}
 
\renewcommand{\theenumi}{(\roman{enumi})}

%

\subsection{Replete pairs}\label{ss:replete_pairs}

This subsection is a generalization of {\cite[Sect.~4.2.]{BPR}}.
We distinguish a class of pairs $W', W''$ with particularly nice properties.

\begin{defin}
We say $(W', W'')$ is \emph{replete with respect to} $v \in \Prime$ (or  $\Bis$), if $v \in W'$ (resp. $v \in W''$). 
   Similarly, we say $(W', W'')$ is \emph{replete} if it is replete with respect to $v$ for all $v \in \Prime \sqcup \Bis$. 
\end{defin}

\begin{rem}
   Strictly speaking, the notion of \emph{replete pair} depends also on the minimal decomposition $V$. 
   But as always we consider a pair $W'$ and $W''$ with a fixed decomposition
     $V=\linspan{V_{\Seg}} \supset W'\oplus W''$, so we refrain from mentioning $V$ in the notation.
\end{rem}

The first important observation is, that as long as we look for pairs that fail to satisfy the additivity, we are free to replenish any pair. 
More precisely, for any fixed $W'$, $W'', V$ and $v \in \Prime$ (or $\Bis$) define the \emph{repletion} of $(W', W'')$ with respect to $v$ as the pair $(\repletionV{W'},\repletionV{W''})$:
\begin{equation}
\begin{aligned}
   \repletionV{W'}:&=W'+\linspan{v}, & 
   \repletionV{W''}:&=W'', &  
   \repletionV{W}:&=\repletionV{W'}\oplus \repletionV{W''}. \\
    \text{or resp.} & &  \\
   \repletionV{W'}:&=W', & 
   \repletionV{W''}:&=W'' +\linspan{v}, &  
   \repletionV{W}:&=\repletionV{W'}\oplus \repletionV{W''}.
\end{aligned}
\end{equation}
The result of consecutive repletion with respect to all elements of $\Prime$ and $\Bis$ will be denoted by $(\repletion{W'},\repletion{W''})$. 
This latter notion agrees with one introduced in  \cite[Subsect. 4.2]{BPR}.

\begin{prop}\label{p:does_not_hurt_to_replenish_v}
   For any $(W', W'')$ and $v \in \Prime$ (or $\Bis$)
we have:
   \begin{align*}
        R(W' ) \le R(\repletionV{W' }) & \le R(W' ) + (\dim \repletionV{W' } - \dim W' ),\\
        R(W'') \le R(\repletionV{W''}) & \le R(W'') + (\dim \repletionV{W''} - \dim W''),\\
        R(\repletionV{W}) & = R(W).
   \end{align*}
   In particular, if the additivity of the rank fails for $(W',W'')$, then it also fails for 
      $(\repletionV{W'},\repletionV{W''})$. 
   Moreover,
   \begin{enumerate}
    \item \label{it:repletion_has_the_same_decomposition_v} 
          $V$ is a minimal decomposition of $\repletionV{W}$; 
             in particular, the same distinguished basis 
             $\Prime\sqcup \Bis \sqcup\dotsb\sqcup \Mix$ works for both $W$ and  $\repletionV{W}$.
    \item \label{it:repletion_is_replete_v}
          $(\repletionV{W'}, \repletionV{W''})$ is a replete pair with respect to $v$.
    \item \label{it:gap_is_preserved_under_repletion_v}
          The gaps 
          $\gap{\repletionV{W'}}$, $\gap{\repletionV{W''}}$, and  $\gap{\repletionV{W}}$, 
               are at most (respectively) 
               $\gap{W'}$, $\gap{W''}$, and  $\gap{W}$.
   \end{enumerate}
\end{prop}

\begin{proof}
    Since $W' \subset \repletionV{W'}$, the inequality $R(W')\le R(\repletionV{W'})$ is clear.
    Moreover, $\repletionV{W'}$ is spanned by $W'$ if $\dim \repletionV{W' } = \dim W' $, or by $W'$ with additional matrix $v$ in the other case. 
       The matrix $v$ is of rank one, so $R(\repletionV{W'}) \le  R(W') + (\dim \repletionV{W' } - \dim W' )$.
    The inequalities about ${}''$ and $R(W)\le R(\repletionV{W})$ follow similarly.
    
    Further $\repletionV{W} \subset V$, thus $V$ is a decomposition of $\repletionV{W}$.
    Therefore also $R(\repletionV{W}) \le \dim V = R(W)$, showing $R(\repletionV{W})  = R(W)$ 
       and \ref{it:repletion_has_the_same_decomposition_v}.
    Item~\ref{it:repletion_is_replete_v} follows from \ref{it:repletion_has_the_same_decomposition_v},
       while \ref{it:gap_is_preserved_under_repletion_v} is a rephrasement of the initial inequalities.
\end{proof}

Moreover, if one of the inequalities of Lemma~\ref{lemma_bound_r'_e'_R_w'} is an equality, 
    then the respective $W'$ or $W''$ is not affected by the repletion.

\begin{lemma}\label{lem_if_inequality_then_Wbis_is_replete}
    If, say, $R(W') + \bfe'' = R(W) - \dim W''$, then for any $v \in \Prime$ (or $\Bis$) we have $W'' = \repletionV{W''}$. The analogous statements hold for the other equalities coming from replacing $\le$ by $=$ in 
       Lemma~\ref{lemma_bound_r'_e'_R_w'}.
\end{lemma}

\begin{proof}
    By Lemma~\ref{lemma_bound_r'_e'_R_w'} applied to $\repletionV{W} = \repletionV{W'}\oplus \repletionV{W''}$ 
      and by Proposition~\ref{p:does_not_hurt_to_replenish_v}
    \begin{align*}
        R(\repletionV{W}) - \bfe'' & \stackrel{\text{\ref{lemma_bound_r'_e'_R_w'}}}{\ge} R(\repletionV{W'}) + \dim (\repletionV{W''}) \\
        & \stackrel{\text{\ref{p:does_not_hurt_to_replenish_v}}}{\ge} R(W') + \dim W'' \\ 
        &\stackrel{\text{assumptions of \ref{lem_if_inequality_then_Wbis_is_replete}}}{=}\ \ R(W) - \bfe'' 
        \stackrel{\text{\ref{p:does_not_hurt_to_replenish_v}}}{=} R(\repletionV{W}) - \bfe''.
    \end{align*}
    Therefore all inequalities are in fact equalities. In particular,  $\dim (\repletionV{W''}) = \dim W''$.
    The claim of the lemma follows from $W'' \subset \repletionV{W''}$.
\end{proof}

As a corollary we can prove, that if $R(W'') \le \dim W'' +2 $, then either rank additivity holds or $W''$ is equal to its repletion.
\begin{cor}\label{cor_if_diff_betw_R_and_dim_Wbis_eq_2_then}
    Assume $R(W'') \le \dim W'' +2 $. Then either the additivity holds $R(W) = R(W') + R(W'')$ 
       or:
    \begin{itemize}
     \item $R(W'') = \dim W'' +2 $, and 
     \item $R(W) = R(W') + R(W'')-1$, and 
     \item $\bfe'' = \bff''=1$, and 
     \item $\repletion{W''} = W''$.
    \end{itemize}
\end{cor}

\begin{proof}
   Assume, that the additivity does not hold. 
   Then by Lemma~\ref{lem_rank_at_least_2_more_than_dimension} we must have $R(W'') = \dim W'' +2 $. 
      By Proposition~\ref{prop_SAC_if_E'=0} follows $\bfe''> 0$, $\bff''> 0$, while by Corollary~\ref{cor_bounds_on_es_and_fs} 
      we obtain $\bfe''< 2$ and $\bff'' < 2$.
   Thus $\bfe'' =\bff'' = 1$.
  
   By Lemma~\ref{lemma_bound_r'_e'_R_w'} the inequality $R(W) \ge R(W')+ 1 + \dim W''$ holds.
   The right hand side is equal to $R(W') + R(W'')-1$ by the above discussion 
      (the $\le$ inequality follows from the failure of additivity).

   The final claim $\repletion{W''} = W''$ follows from Lemma~\ref{lem_if_inequality_then_Wbis_is_replete}.
\end{proof}

Later, in Corollary~\ref{cor_1_2_hook_shaped} we will show that if the difference between rank and dimension of $W''$ is at most two, then  rank additivity holds. 
  
\subsection{Digestion with respect to a rank one tensor}\label{ss:digestion}
This subsection is a generalization of {\cite[Sect.~4.2.]{BPR}}.
For pairs which are replete with respect to $v \in \Prime$ (or $\Bis$) it makes sense to consider the complement of $\linspan{v}$ in $W'$ (resp. $\linspan{v}$ in $W''$).

\begin{defin} Using Notation~\ref{notation_Prime_etc}, let $v \in \Prime \sqcup \Bis$ 
    and $\digestionV W'$, $\digestionV W'' $ denote the following linear spaces:
    \begin{equation*}
     \left\{
    \begin{aligned}
\digestionV W' :=&\linspan{\ccB \setminus \{v\}}
         \cap W',  &
      \digestionV W'' :=&W'',
    & \text{if } &v \in \Prime, \\
\digestionV W' :=& W',  &
      \digestionV W'':=&\linspan{\ccB \setminus \{v\}}
          \cap W'', 
    & \text{if } &v \in \Bis.
\end{aligned}
    \right.
\end{equation*}
 
  We call the pair $(\digestionV W' , \digestionV W'')$ the \emph{digested version of $(W',W'')$ with respect to} $v$.  Similarly, by $(\digestion W' , \digestion W'' )$ we will denote the result of consecutive digestion with respect to
all elements of $\Prime$ and $\Bis$. This latter notion agrees with one introduced in  \cite[Subsect. 4.3]{BPR}.
\end{defin}

\begin{lemma}\label{lem_dim_S'_v}
 If $(W', W'')$ is replete with respect to $v \in \Prime$ (or $\Bis$),
   then $W'  = \linspan{v} \oplus \digestionV {W'}$ and $W''  = \digestionV {W''}$. (resp.  $W'  =  \digestionV {W'}$ and $W''  = \linspan{v} \oplus \digestionV {W''}$).
\end{lemma}
\begin{proof}
We will prove only case when $v \in \Prime$. The case when $v \in \Bis$ is similar.
Both $\linspan{v}$ and $\digestionV {W'}$ are contained in $W'$. 
The intersection $\linspan{v} \cap \digestionV {W'}$ is zero, since 
   the seven sets $\Prime, \Bis, \HR, \HL, \VL, \VR, \Mix$ are disjoint and together they are linearly independent.
Furthermore,  
\begin{align*}
   \codim (\digestionV {W'} \subset W') \le &  \\
    \codim (\linspan{(\Prime \setminus \{v\}) \sqcup \Bis \sqcup \HL \sqcup \HR \sqcup \VL \sqcup \VR \sqcup \Mix } \subset V) = & 1.
\end{align*}
Thus $ \dim W' \leq \dim \digestionV {W'} + 1 $, which concludes the proof. 
\end{proof}

These complements $(\digestionV {W'}, \digestionV {W''})$ might replace the original pair $(W',W'')$ replete  with respect to $v \in \Prime$ (or $\Bis$):
  as we will show in Lemma~\ref{lem_can_digest_v}, if the additivity of the rank fails for $(W',W'')$, it also fails for $(\digestionV {W'}, \digestionV {W''})$.
Moreover, $(\digestionV {W'}, \digestionV {W''})$ does not involve  $v \in \Prime$ (or $\Bis$).
The opposite implication is not true as Lemma~\ref{lem_no_additivity_equivalence_in_digestion} states. 

\begin{lemma}\label{lem_can_digest_v}
Suppose $(W', W'')$ is replete with respect to $v \in \Prime$ (or $\Bis$), define $S':=\digestionV {W'}$ and $S'':=\digestionV {W''}$ and set $S=S'\oplus S''$.
Then
\begin{enumerate}
 \item \label{it:can_digest___rank_v}
       $R(S)=R(W) -1$
         and the space $\linspan{(\Prime \setminus {v}), \Bis, \HL,\HR,\VL,\VR,\Mix}$ determines a minimal decomposition of $S$.
\item \label{it:can_digest___additivity_v}
       If the additivity of the rank $R(S) =  R(S')+ R(S'')$ holds for $S$,
         then it also holds for $W$, that is $R(W) =  R(W')+ R(W'')$.
\end{enumerate}
\end{lemma}
\begin{proof}
We will prove only case when $v \in \Prime$. The case when $v \in \Bis$ is similar.
     By Lemma~\ref{lem_dim_S'_v} we have $W = S\oplus \linspan{v}$, thus $R(W)\le R(S)+ 1$.
     On the other hand, $S\subset \linspan{(\Prime \setminus {v}), \Bis, \HL,\HR,\VL,\VR,\Mix}$, hence $R(S)\le R(W)-1$.
     These two claims show the equality for $R(S)$ in \ref{it:can_digest___rank_v} 
        and that $\linspan{(\Prime \setminus {v}), \Bis, \HL,\HR,\VL,\VR,\Mix}$ gives a minimal decomposition of $S$.

     Finally, if $R(S) =  R(S')+ R(S'')$, then:
     \begin{align*}
        R(W) &= R(S)+ 1 = R(S')+ R(S'')+1 \ge R(W') + R(W''),
     \end{align*}
     showing the statement  \ref{it:can_digest___additivity_v} for $W$.
\end{proof}

\begin{lemma}\label{lem_no_additivity_equivalence_in_digestion}
Assume that there exist a counterexample to additivity of tensor rank over a base field $\kk$, for example $\kk=\CC$ (see Theorem~\ref{t:shitov_counter_example}). Then, there exists an example of a pair $(W', W'')$ of linear spaces over $\kk$ such that:
\begin{enumerate}
 \item $(W', W'')$ is replete  with respect to $v \in \Prime$, 
 \item the additivity of the rank holds for $W=W' \oplus W''$,
 \item the additivity of the rank does not hold for $S=S' \oplus S''$, where $S':=\digestionV {W'}$, $S'':=\digestionV {W''}$.
\end{enumerate}
\end{lemma}
\begin{proof}
 Assume conversely, that additivity of the rank $R(S) =  R(S')+ R(S'')$ holds for $S$ if and only if it holds for $W$, that is $R(W) =  R(W')+ R(W'')$.
 
 Then take a pair $(W', W'')$ such that $R(W' \oplus W'') <  R(W')+ R(W'')$ and a minimal basis $\mathcal{B}'$ of rank one matrices such that $W'\subseteq \linspan{\ccB'} $. We construct $\tilde{W'}_1$ by adding an element $v \in \mathcal{B}'$ to $W'$. Let us observe, that $R(\tilde{W'}_1)=R(W')$. 
 Indeed,  
$\tilde{W'}_1\subseteq \linspan{\ccB}$ implies $R(\tilde{W'}_1) \leq R(W')$
 The opposite inequality follows from the fact, that $W' \subseteq \tilde{W'}_1$.
 If $R(\tilde{W'}_1 \oplus W'') =R(\tilde{W'}_1) + R(W'') $ we have a contradiction, because we can always choose a basis for $\tilde{W'}_1 \oplus W''$ and partition $\Prime, \Bis,..., \Mix$ in  a way that $v \in \Prime$.  Thus, we may assume the right hand side is smaller. 
 
 We repeat the process with $v_1 \neq v_2 \in \mathcal{B}'$ and $\tilde{W'}_1$ in place of $W'$, obtaining subspace $\tilde{W'}_2$. We do it inductively. We denote by $n$ the smallest number $i$ such that $\mathcal{B}' \subseteq \tilde{W'}_i$. 
 As discussed before, we must have $R(\tilde{W'}_n \oplus W'') <R(\tilde{W'}_n) + R(W'') $
 and we may assume that in the minimal decomposition of  $\tilde{W'}_i \oplus W''$ all matrices from $\mathcal{B}'$ belong to $\Prime$.
 After the process of digestion of all Primes of $\tilde{W'}_n \oplus W''$ we obtain $\emptyset \oplus W''$ for which rank additivity trivially holds. Thus from Lemma~\ref{lem_can_digest_v} \ref{it:can_digest___additivity_v} we know, that $R(\tilde{W'}_n \oplus W'') =R(\tilde{W'}_n) + R(W'') $, a contradiction.
\end{proof}

As a summary, in our search for examples of failure of the additivity of the rank,
   in the previous section we replaced a linear space $W=W'\oplus W''$ by its
repletion with respect to $v \in \Prime$ (or $\Bis$) $\repletionV{W} = \repletionV{W'} \oplus \repletionV{W''}$, that is possibly larger.
Here in turn, we replace  $\repletionV{W}$ by a smaller linear space $S := S' \oplus S''$, where $S' :=\digestionV {(\repletionV{W'})}, S'' := \digestionV {(\repletionV{W''})}$.
In fact, $\dim S' \leq \dim W'$ and $\dim S'' \leq \dim W''$, and also $R(S)\le R(W)$.
That is, changing $W$ into $S$ neither makes the corresponding tensors larger nor decreases the defect.
   
\begin{cor}\label{cor_can_replete_and_digest}
 Let $(W', W'')$ be a pair of linear spaces, $v \in \Prime$ and $(S',S''):=(\digestionV {(\repletionV{W'})}, \digestionV {(\repletionV{W''})})$. Then the following inequalities holds:
 \begin{enumerate}
  \item $0 \leq \dim \repletionV{W'} - \dim W' \leq 1$, \label{it_cor_can_RD_1}
  \item $\dim S' = \dim \repletionV{W'} -1$, \label{it_cor_can_RD_2}
\item $\dim S'' = \dim W''$, \label{it_cor_can_RD_3}
  \item $R(S)= R(W) -1$ and the space $\linspan{(\Prime \setminus {v}), \Bis, \HL,\HR,\VL,\VR,\Mix}$ determines a minimal decomposition of $S$, \label{it_cor_can_RD_4}
  \item $R(W')-1 \leq R(S')\leq R(W') + (\dim \repletionV{W'} - \dim W')$, \label{it_cor_can_RD_5}
  \item $S''= W''$, in particular $R(S'')= R(W'')$. \label{it_cor_can_RD_6}
 \end{enumerate}
 Moreover, the defect does not decrease after the process of repletion and digestion by $v$.
 In particular if the rank additivity does not hold for $(W', W'')$, then for $(S',S'')$ does not hold as well.
\end{cor}
\begin{proof}
 The proof follows directly from Proposition~\ref{p:does_not_hurt_to_replenish_v},  Lemma~\ref{lem_dim_S'_v} and Lemma~\ref{lem_can_digest_v}.
\end{proof}

The following observation states, that after repletion and digestion with respect to all elements of $\Prime$ there is no $a'$ in $A'^*$ such that the slice $p(a')\in B' \otimes C'$ is of rank one. 
   
\begin{lem}\label{lem_no_rank_1_slice_after_digestion}
 Suppose $p(A^*)=W' \oplus W'' \subseteq (B'\oplus B'') \otimes (C' \oplus C'')$, where $W'$ is equal to its digested and repleted version with respect to all elements of $\Prime$. Then, there is no $a' \in A'^*$ such that $v:=p(a')\in B' \otimes C'$ is a rank 1 matrix.
\end{lem}
   \begin{proof}
    Let us assume the opposite, there exists $a' \in A'^*$ such that $v\in B' \otimes C'$ is a rank 1 matrix. It follows from the assumption about $W'$ that $\Prime=\emptyset$, 
    thus $v \in \langle \HL, \VL, \HR, \VR, \Bis, \Mix \rangle \setminus \langle \Bis \rangle$. It is a contradiction with the way we partition the basis
    $\ccB$ 
to $\Prime, \Bis, ..., \Mix$ see Notation~\ref{notation_Prime_etc}.
   \end{proof}

We may replenish and digest also in the other directions. It turns out, that we can precisely say what happens with the hook-structure (Definition~\ref{def_hook_shaped_space}) when we choose the repletion vector wisely. 

\begin{cor}\label{c:removing_rank_1_slice}
 Assume that $p(A'^*)=W' \subset B' \otimes C'$ is $(\mathbf{k},\mathbf{l})$-hook shaped, i.e. there exists $G' \subset B', H' \subset C'$ such that $W' \subset G' \otimes C' + B' \otimes H'$, where $dim(G')=\mathbf{k}, \dim(H')=\mathbf{l}$. Assume further, that there is $\gamma \in (C'/H')^*$ such that $v:=p(\gamma)$ is a rank 1 matrix. Then after the process of repletion and digestion of $p(C^*)$ with respect to $v$ we obtain tensor $\tilde{p}$ such that:
 \begin{enumerate}
  \item $\tilde{p}=\tilde{p}' \oplus p''$, \label{c:repl_hook_i}
  \item  $ \tilde{p}' \in \tilde{A}'\otimes B' \otimes \tilde{C}'$,
  where 
  $\tilde{A}' \subseteq A'$ is such that $\tilde{p}'$ is $\tilde{A}'$-concise,
  $\tilde{C}':=\gamma^{{\perp}}\subset C'$ is the linear hyperplane $(\gamma =0)$, 
  and $R(\tilde{p}) = R(W)-1$, \label{c:repl_hook_ii}
 \item  $\tilde{p}(A'^*)$ is still $(\mathbf{k},\mathbf{l})$-hook shaped, \label{c:repl_hook_iii}
 \item  If the additivity of the rank does not hold for $p$,
         then it also does not hold for $\tilde{p}$, \label{c:repl_hook_iv}
 \item  If $p'$ is $A'$-concise, then $\dim A' - 1\leq  \dim \tilde{A}'$,\label{c:repl_hook_v}
 \item  If $p'$ is $A'$-concise and the set $\Prime$ in the decomposition of $p(A^*)$ is empty, then $ A' =  \tilde{A}' $.\label{c:repl_hook_vi}
 \end{enumerate}
\end{cor}
\begin{proof}
For  $p(C^*) \subseteq A \otimes B$ we choose a minimal decomposition $V_C= \linspan{V_{C,\Seg}} \subset A\otimes B$ and $\fromto{\Prime_C, \Bis_C,}{\Mix_C}$
     are as in Notation~\ref{notation_Prime_etc} (with added the subscript ``${}_{C}$'' to stress that $B\otimes C$ is changed to $A\otimes B$).
Since $v$ is a rank one matrix and is contained in $A' \otimes B'$, we can choose a minimal decomposition such that $v \in \Prime_C$.
 The process of repletion with respect to $v$ brings no change, because $v$ is already contained in $p(C^*)$.
 
 The matrix $v $ is contained in $A' \otimes G'$, so it is $(0,\mathbf{k})$-hook shaped. In the process of digestion with respect to $v$ we obtain the new tensor $\tilde{p}$ such that $\tilde{p}(C''^*)= \tilde{p}''(C''^*)= {p''(C''^*)}$ and $\tilde{p}(C'^*)=\tilde{p}'(C'^*)$ which differs from $ {p'(C'^*)}$ only in the places corresponding to $A' \otimes G'$. We conclude \ref{c:repl_hook_i} and  \ref{c:repl_hook_iii}. Items \ref{c:repl_hook_ii} and \ref{c:repl_hook_iv} follow from Corollary~\ref{cor_can_replete_and_digest}. 
 
Let us assume that $p'$ is $A'$-concise. Observe, that 
$p'=\tilde{p}'+v\otimes (C'/\tilde{C}') \in A' \otimes B' \otimes C'$.
 Since $v$ is a rank 1 matrix, then either $v \in \tilde{A}' \otimes B'$ or there exists $a \in A' \setminus \tilde{A}'$ such that $v \in \langle\tilde{A}', a\rangle \otimes B'$ (see Figure \ref{fig_digestion_concise}).  In the first case $A' =  \tilde{A}' $. In the second case  $\langle\tilde{A}', a\rangle =A'$. We obtained \ref{c:repl_hook_v}. Let $0 \neq \alpha\in A'^*$ be 
such that $\alpha^\perp = \tilde{A}'$.
 Then, $p'(\alpha)\in B' \otimes C'$ is a rank one matrix contained in $B' \otimes (\gamma=1)$. If $\Prime_A = \emptyset$, then we have a contradiction with Lemma~\ref{lem_no_rank_1_slice_after_digestion}. We proved \ref{c:repl_hook_vi}.
 
 \begin{figure}
\centering
\includegraphics[scale=0.06]{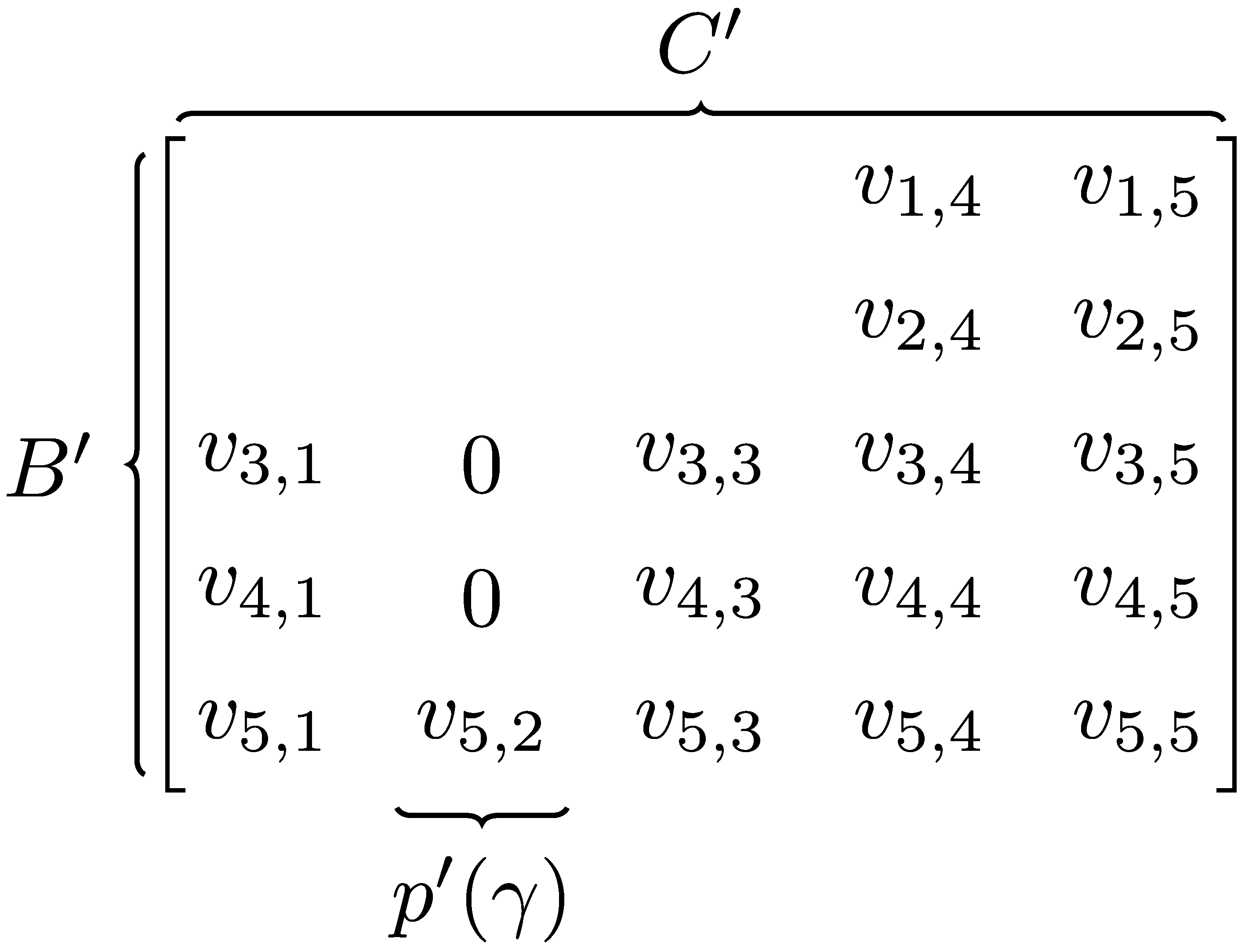}
\includegraphics[scale=0.06]{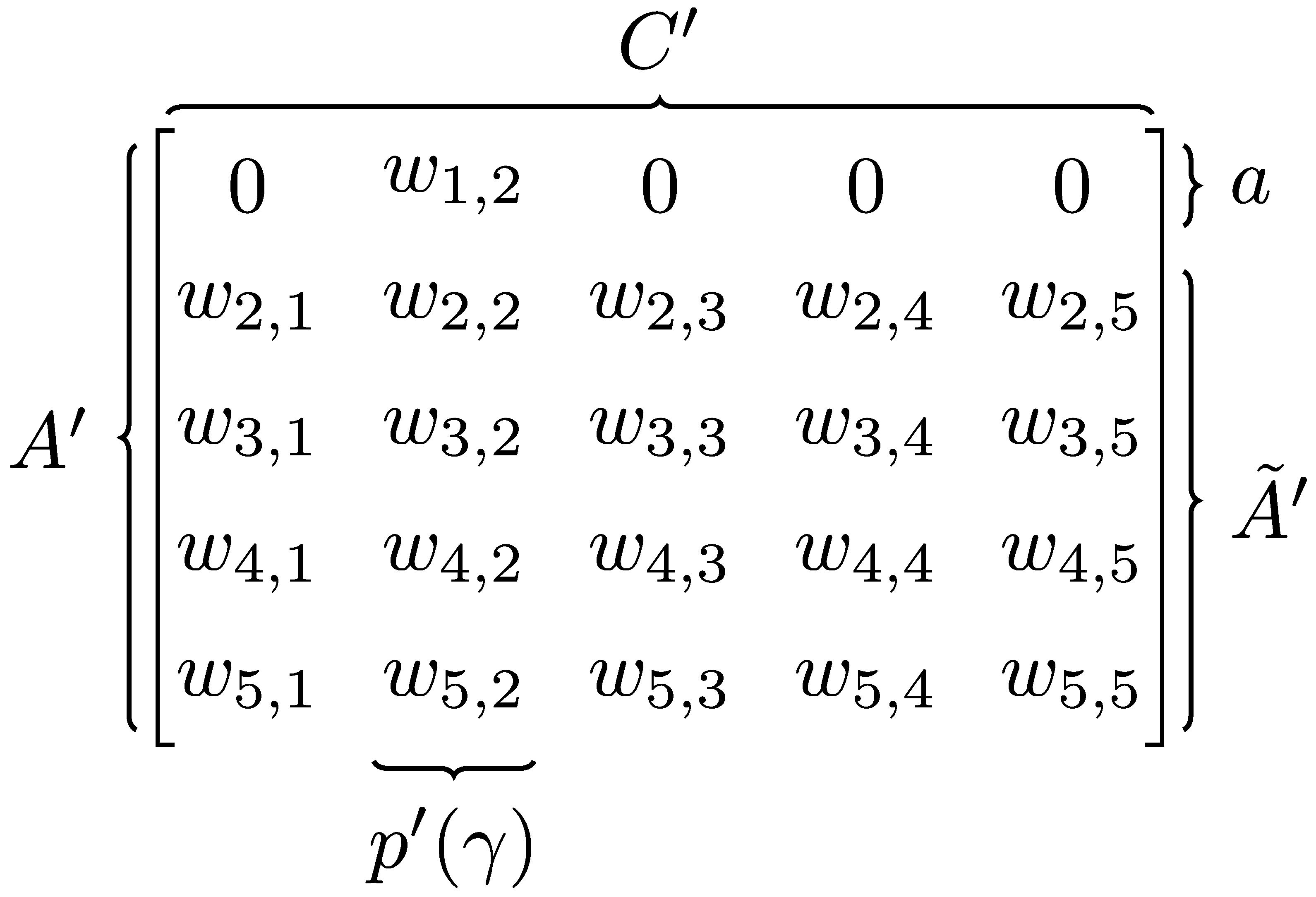}
\caption{Let a tensor $p=p'+p'' \in (A'\oplus A'') \otimes (B'\oplus B'') \otimes (C'\oplus C'')$, where $\dim(A',B',C')=(5,5,5)$ and $p(A'^*)$ is $(3,2)$-hook shaped. At the figure, there are shown sample spaces of slices $p'(A'^*)$ and $p'(B^*)$. Zero elements are denoted either by a blank space or explicitly by 0.  Notice, that $p'$ is $A'-concise$. We replete and digest $p(C^*)$ with respect to $p(\gamma)$, where
$\gamma \in (C'/H')^*$.
In result we obtain tensor $\tilde{p'} \oplus p''$ such that $p'=\tilde{p}'+p'(\gamma)\otimes \gamma \in A' \otimes B' \otimes C'$.
Since $p(\gamma)$ is a rank 1 matrix, then either $v \in \tilde{A}' \otimes B'$ or there exists $a \in A' \setminus \tilde{A}'$ such that $v \in \langle\tilde{A}', a\rangle \otimes B'$. In the second case, let $\alpha\in A'^*$ be dual to $a$. Then $p'(\alpha)\in B' \otimes C'$ is a rank one matrix contained in $B' \otimes (\gamma=1) \subseteq B' \otimes C'$.
} \label{fig_digestion_concise}
\end{figure}
\end{proof}

It follows from Corollary~\ref{c:removing_rank_1_slice}, that if in the pair of linear subspaces $(W',W'')$ without the rank additivity property, one of them, say $W'$ is $(1,\mathbf{k})$-hook shaped, then we can construct $\hat{W'}$, which is $(0,\mathbf{k})$ hook shaped and the pair $(\hat{W}',W'')$ does not poses the rank additivity property either.

\begin{prop}\label{p:1_n_hook_shaped}
  Suppose $W'\subset B'\otimes C'$, $\mathbf{k}<\bfc'$, $W'$ is $(1,\mathbf{k})$-hook shaped and $W''\subset B''\otimes C''$ 
      is an arbitrary subspace.
  If the additivity of the rank fails for $W'\oplus W''$, then it also fails for smaller subspaces $\hat{W}',W''$, where $\hat{W}' \subset B' \otimes \kk^{\mathbf{k}}$.
\end{prop}
\begin{proof}
We can assume that $W'$ is concise. It is straightforward to verify, that for every $\gamma'^*\in (C'/F')^* \subset (C')^*$ we have $p(\gamma'^*) = p'(\gamma'^*) \in A' \otimes B'$ has rank $1$.
        Then from Corollary~\ref{c:removing_rank_1_slice}, the process of repletion and digestion with respect to $p(\gamma'^*) $
leads to another pair $\widetilde{W}',W''$, where  $\widetilde{W}'\subset B'\otimes \kk^{\bfc'-1}$. Subspace $\widetilde{W}'$ is again $(1,\mathbf{k})$-hook shaped. The new pair is also a counterexample to the additivity of the rank, if the starting pair was. 
        
        By a consecutive repeating the process for another $\bfc'-\mathbf{k}-1$ times, we shrink $W'$ to $(0,\mathbf{k})$-hook shaped $\hat{W'}$. Together with $W''$ it creates the desired pair $(\hat{W}',W'')$ from the statement of the proposition.
\end{proof}

Proposition~\ref{p:1_n_hook_shaped} allows us to generalize both \cite[Proposition 3.17]{BPR} and Theorem~\ref{t:jaja_2_0_hook}. The latter one can be thought of as a theorem about $(0,2)$-hook shaped spaces. Observe, that we do not need the base field to be algebraically closed.
   
\begin{cor}\label{cor_1_2_hook_shaped}
Let $W =  W'\oplus W''$ 
where $W'$ is a $(1,2)$-hook shaped,
then the additivity of rank $R(W) =  R(W')+ R(W'')$ holds.
\end{cor}
\begin{proof}
   We use Proposition~\ref{p:1_n_hook_shaped} to reduce the problem to the case when $W' \subseteq B' \otimes C'$ is $(0,2)$-hook shaped. Then translating Theorem~\ref{t:jaja_2_0_hook} to a language of subspaces (see Lemma~\ref{lem_rank_of_space_equal_to_rank_of_tensor}), we obtain that the pair $(W', W'')$ has rank additivity property. 
\end{proof}
For a future reference, we state what we know about the case when 
the subspace $\Mix$ is not concise in $(E' \oplus E'') \otimes (F' \oplus F'')$. Say $E''$ can be replaced by a smaller one, $\widetilde{E}''$ such that $\Mix \subseteq (E' \oplus \widetilde{E}'') \otimes (F' \oplus F'')$. To make the proof clearer, this time we assume that $W''$ is hook shaped.

\begin{lemma}\label{lem_mix_niedomaga}
Assume $\Bis=\emptyset$ 
 and $\mathbf{k}$ is the smallest natural number such that 
$W''$ is $(\mathbf{k},\dim(F''))$-hook shaped. Let $\widetilde{E}''\subseteq E''$ and $\widetilde{F}'\subseteq F'$ be the smallest subspaces such that $\Mix \subset (E' \oplus \widetilde{E}'') \otimes (\widetilde{F}' \oplus F'')$. 
The additivity of ranks $R(W) =  R(W')+ R(W'')$ holds if all of the following conditions are fulfilled:
 \begin{enumerate}
  \item $\dim(\widetilde{E}'') \leq \mathbf{k} -1$, \label{it:mix_niedomaga}  
\item $\pi_{E' \oplus {B''}}(\VL)$ is linearly independent and concise in $(B' / E') \otimes F'$, \label{it:pi_vl_jest_lnz}
  \item $\pi_{{B'} \oplus \widetilde{E}''}\pi_{C''}(\HR)$ is linearly independent. \label{it:hr_lnz}
\end{enumerate}
\end{lemma}
\begin{proof}
 Let us assume by contradiction, that additivity of ranks does not hold. 
 Firstly we show that we can assume that there is no $\Prime$ or $\Bis$ in the decomposition of $W=W'\oplus W'' :=(p'+p'')(A^*)$. If there is, we replete and digest obtaining $\digestion\repletion{W'}, \digestion\repletion{W''}$ and show a contradiction for this new pair.  Everything we need to know to make this assumption is given by Corollary~\ref{cor_can_replete_and_digest}. 
 
 To make it explicit, we have the following facts. By Corollary~\ref{cor_can_replete_and_digest} \ref{it_cor_can_RD_4}, if the space $\langle \Prime, \Bis, \HR, \VR, \VL, \VR, \Mix \rangle$ determines the minimal decomposition for $W'\oplus W''$, then the subspace  $\langle \HR, \VR, \VL, \VR, \Mix \rangle$ determines the minimal decomposition for $\digestion\repletion{W'}\oplus \digestion\repletion{W''}$. Thus, the new pair $\{\digestion\repletion{W'}, \digestion\repletion{W''}\}$ still fulfills conditions from the statement. By Corollary~\ref{cor_can_replete_and_digest},
 if we prove the rank additivity for $\{\digestion\repletion{W'}, \digestion\repletion{W''}\}$, we show it for starting tensors as well, contradicting our assumption.
 
 From condition \ref{it:mix_niedomaga},
there exists an element $w \in W''$ such that $\pi_{B' \oplus \widetilde{E}''}\pi_{C' \oplus F''} (w)$ is nonzero (cf. Figure \ref{fig_E_tilde}). To present $w$ as a linear combination of vectors from $\HL,\HR, \VL,\VR, \Mix$ we need an element $h \in \HR$ such that $\pi_{B' \oplus \widetilde{E}''}\pi_{C' \oplus F''}(h)$ is nonzero. Notice, that $\pi_{C''}(h)$ is nonzero. To get rid of $\pi_{C''}(h)$ in the presentation of $w$ we have to use an element $v \in \langle \VL \rangle$. Indeed, from conditions \ref{it:mix_niedomaga} and \ref{it:hr_lnz} follows that we cannot restrict ourselves to elements from $\Mix$ or $\HR$ for it.  Now, from condition \ref{it:pi_vl_jest_lnz} $\pi_{E'\oplus B''}(v) \neq 0$. Thus $\pi_{E'\oplus B''}(w) \neq 0$, which is a contradiction with the assumption that $w \in W''$.

 \begin{figure}
\centering
\includegraphics[scale=0.06]{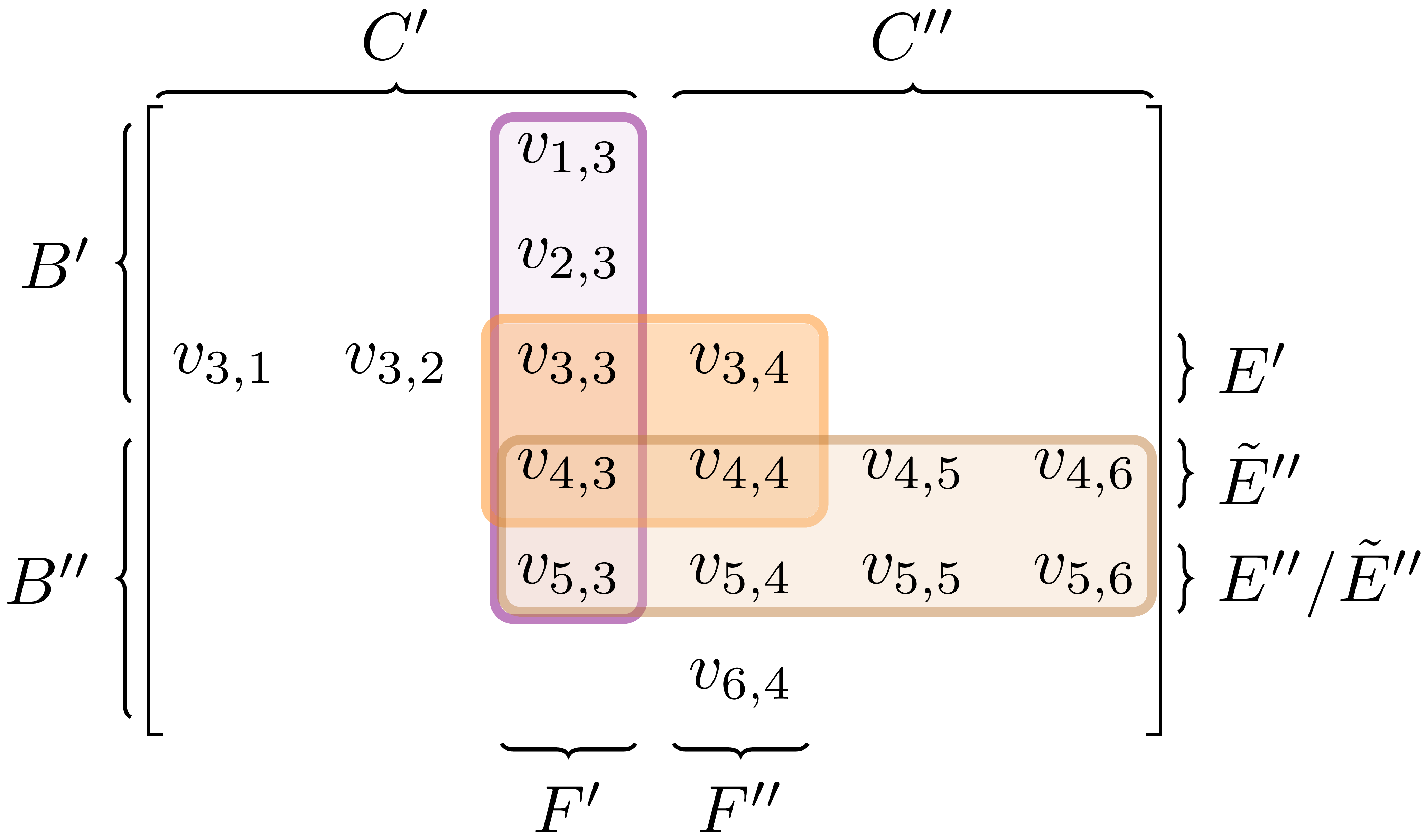}
\caption{We use the notation from Lemma~\ref{lem_mix_niedomaga}. At the figure, there are shown sample spaces of slices $\digestion\repletion{W'}\oplus \digestion\repletion{W''} \subseteq (\kk^3 \oplus \kk^3)\otimes(\kk^3 \oplus \kk^3)$ such that $\mathbf{k}=2$, $\dim(\tilde{E}'')=1$.
 Zero elements are denoted by a blank space.  
  We highlighted subspaces 
\VL{} (purple with entries $v_{1,3},v_{2,3},v_{3,3},v_{4,3},v_{5,3}$), 
\HR{} (brown with entries $v_{4,3},v_{4,4},v_{4,5},v_{4,6},v_{5,3},v_{5,4},v_{5,5},v_{5,6}$),
and \Mix{} (middle orange square with entries $v_{3,3},v_{3,4},v_{4,3},v_{4,4}$).
} \label{fig_E_tilde}
\end{figure}
\end{proof}

\subsection{Additivity of the tensor rank for small tensors}\label{sect_proofs_of_main_results_on_rank}

Assuming the base field is $\kk=\CC$, 
   one of the smallest cases not covered by the Theorem \ref{thm_additivity_rank_intro_jarek} would be
   the case of $p', p''\in \CC^4\otimes \CC^4\otimes \CC^3$.
The generic rank (that is, the rank of a general tensor) in $\CC^4\otimes \CC^4\otimes \CC^3$
   is $6$, moreover \cite[p.~6]{atkinson_stephens_maximal_rank_of_tensors} 
   claims the maximal rank is $7$ (see also \cite[Prop.~2]{sumi_miyazaki_sakata_maximal_tensor_rank}). 
To prove Corollary~\ref{cor_additivity_C_443_443}, i.e rank additivity property for the mentioned cases, we need to establish the following lemma.

\begin{lem}\label{lem_if_rank_3_more_than_dimension}
    Let us use Notation~\ref{notation_Prime_etc}. Assume $\kk$ is an algebraically closed base field, $W'\subseteq \kk^{\bfb'} \otimes \kk^{\bfc'}$ is of dimension $\bfa'$, $W''\subseteq \kk^{\bfb''} \otimes \kk^{\bfc''}$ is of dimension $\bfa''$, the corresponding tensor is $p = p' \oplus p'' \in A' \otimes B' \otimes C' \oplus A'' \otimes B'' \otimes C''$ and:
    \begin{enumerate}
     \item $W'$ is concise,
\item $\Prime=\emptyset$, $\Bis=\emptyset$, 
\item \label{it:eq_lem_if_rank_3_more_than_dimension}
   $\bfa'+3= R(W')$,
     \item \label{it:ineq_lem_if_rank_3_more_than_dimension}
   $R(W') \leq \bfb' + \bfc'$,
\end{enumerate}
   Then either additivity of ranks holds or  all the following conditions are satisfied
   \begin{itemize}
    \item $R(W')+R(W'') -1= R(W)$
    \item $R(W')=\hl+\vl+2$
    \item $R(W'')=\hr+\vr+1$
    \item $\mix=2$ 
   \end{itemize}
\end{lem}
\begin{proof}
    Let us assume the additivity of ranks does not hold.
Corollary~\ref{cor_bounds_on_es_and_fs} and assumption~(iii) imply that both $\bfe', \bff'$ are at most 2.
       From Corollary~\ref{cor_1_2_hook_shaped} and assumption~(ii) follows that 
       $\bfe'=\bff'= 2$.
        
        Now we show that the defect equals one, i.e. $d:=R(W')+R(W'')-R(W)=1$.
        We have from Lemma~\ref{lemma_bound_r'_e'_R_w'} that $R(W'') + \bfe' + \bfa' \leq R(W)$, thus $R(W')+R(W'')-1 \leq R(W)$ by the condition \ref{it:eq_lem_if_rank_3_more_than_dimension}.
        The opposite inequality follows from the failure of the rank additivity.

        Since $\Prime=\emptyset$, we must have 
        $    W'\subset \linspan{\pi_{C''}(\HL), \pi_{B''}(\VL), \pi_{B''}\pi_{C''}(\Mix)} 
            \subset {E'}  \otimes C' + B'\otimes F'.
        $
        That is, $W'$ is $(2,2)$-hook shaped.
        
It follows from Proposition \ref{p:1_n_hook_shaped} that there are no integers $n<\bfb'$, $m<\bfc'$ such that $W'$ is $(n,1)$-hook shaped or $(1,m)$-hook shaped.  Moreover, we can assume that
\begin{enumerate}[label=(\alph*)]
   \item \label{it:beta_lem_if_rank_3_more_than_dimension} 
   there is no $\beta \in (B'/E')^*$ such that $p(\beta)$ is a rank one matrix, and 
   \item \label{it:gamma_lem_if_rank_3_more_than_dimension}
   there is no $\gamma \in (C'/F')^*$ such that $p(\gamma)$ is a rank one matrix. 
\end{enumerate}

        Next, we show that 
        \begin{equation}\label{eq:bound_on_vl}
        \bfb'-1  \leq \vl.       
        \end{equation}
        For this purpose we consider the projection $\pi_{E'\oplus B''}\colon B \to B'/E'$.
        The related map $B\otimes C \to   (B'/E')\otimes C$ 
        (which by the standard abuse we also denote $\pi_{E'\oplus B''}$), 
        kills all the rank one tensors of types $\HL$, $\HR$, $\VR$ and $\Mix$, possibly
        leaving a few of type $\VL$ alive.
        The image $\pi_{E'\oplus B''}(W) \subset  (B'/E')\otimes F'$ has rank at most $\vl$ 
        and is concise (otherwise, either
        there is  $\beta \in (B'/E')^*$ such that $p(\beta)$ is a rank one matrix
or  $p'$ is not concise, a contradiction in both cases).
        Note, that $(B'/E')\otimes F' \simeq \CC^{\bfb'-2} \otimes \CC^2$.
        Concise linear subspaces $G$ of $\CC^{\bfb'-2} \otimes \CC^2$ need to have a rank at least $\bfb'-2$. 
        
        In the case, when the rank of $G$ is exactly $\bfb' -2$. The tensor $g$ corresponding to  $G$, is contained in the space $\CC^{\bfb' -2} \otimes (B'/E')\otimes F' \simeq \CC^{\bfb' -2} \otimes \CC^{\bfb' -2} \otimes \CC^2$. It follows from 
Lemma~\ref{lem_dim_2_have_rank_one_matrix} that
there exists $\beta'\in (B'/E')^* \subset B'^*$
such that $g(\beta')$ has rank 1.
       Furthermore, both tensors $p'(\beta') \in A' \otimes C'$ and $p(\beta')  \in A \otimes C$ have rank $1$ as well, contradicting \ref{it:beta_lem_if_rank_3_more_than_dimension}.
Thus, $R(\pi_{E'\oplus B''}(W))$ must be at least $\bfb'-1$ and consequently, $\bfb'-1 \leq \vl$. 
Analogously, we can prove
                \begin{equation}\label{eq:bound_on_hl}
                \bfc'-1  \leq \hl.
                \end{equation}

        By Proposition~\ref{prop_projection_inequalities}\ref{item_projection_inequality_short_bis} we obtain
        $
        R(W'')\leq \hr +\vr +\mix = R(W) - (\hl+\vl) .
        $ Thus
        \begin{equation}\label{eq:bound_on_sum_hl_vl}
          \hl+\vl\leq R(W') -1.
        \end{equation}
        and similarly 
        \begin{equation}\label{eq:bound_on_sum_hr_vr}
          \hr+\vr\leq R(W'') -1.
        \end{equation}

        At least one of $\pi_{E' \oplus B''}(\VL)$, $\pi_{F' \oplus C''}(\HL)$ is linearly independent. Otherwise, arguing as before, we see that
         $\pi_{E'\oplus B''}(W)$ has rank at most $\vl-1$. Thus,          
        $\bfb'-1\leq \vl-1$ and similarly $\bfc'-1\leq \hl-1$. 
        Together with an inequality \ref{eq:bound_on_sum_hl_vl} it gives $\bfb' + \bfc' \leq \hl + \vl \leq R(W') -1 \leq \bfb'+\bfc' -1$, a contradiction. The last inequality is implied by the condition \ref{it:ineq_lem_if_rank_3_more_than_dimension}.

        Now we will prove, that $R(W')\leq \hl + \vl+2$.
        Let us assume   $\pi_{E' \oplus B''}(\VL)$ is linearly independent. 
        (In the case when $\pi_{F' \oplus C''}(\HL)$ is linearly independent we proceed similarly by exchanging $B'$ with $C'$ and $\HL$ with $\VL$). 
        It follows 
        \begin{equation}\label{eq:niedomaga_mix_hr,vr,mix}
        W''= \pi_{B'} (W'') \subseteq \pi_{B'}(\langle \HR,\VR,\Mix\rangle)
        \end{equation}
        and $W'' \not \subseteq \pi_{B'}(\langle \HR,\VR\rangle)$ because of (\ref{eq:bound_on_sum_hr_vr}). 
We obtain from Lemma~\ref{lem_bound_on_rank_for_non_concise_decompositions} and (\ref{eq:niedomaga_mix_hr,vr,mix}) that
         $R(W'') +1 \leq \hr +\vr +\mix$.
        Together with $\hl+\hr+\vl+\vr+\mix =  R(W') + R(W'')-1$ (because defect $d=1$) it gives $\hl + \vl \leq R(W') -2$. 
        
        Assumptions~\ref{it:eq_lem_if_rank_3_more_than_dimension}, \ref{it:ineq_lem_if_rank_3_more_than_dimension} and Equations \eqref{eq:bound_on_vl}, \eqref{eq:bound_on_hl} imply that
\[ R(W') -2 = \bfa'+ 1 \leq \bfb' + \bfc' -2 \leq \hl + \vl. \] 
        Hence, the inequalities can be changed into equalities. In particular $R(W')=\hl + \vl+2$, thus  \begin{equation}\label{eq:inside_proof_rank_of_W''}
        R(W'')=\hr + \vr + \mix -1. 
        \end{equation}        
        In result $\mix \geq 2$. Indeed, from Corollary~\ref{cor_inequalities_HL__Mix} follows that $\mix \geq 1$. If we assume $\mix = 1$, then we look at the $\pi_{B''}\pi_{C''} (W) = W''$, which is contained in $\langle \pi_{B''}\pi_{C''}(\HL),\pi_{B''}\pi_{C''}(\VL),\pi_{B''}\pi_{C''}(\Mix) \rangle$. Taking the rank into account we obtain $R(W'') \leq R(W)-\vr -\hr = R(W)-R(W')$. It is against our assumption saying, that the additivity of ranks does not hold.
        
        We get from \eqref{eq:bound_on_sum_hr_vr} and \eqref{eq:inside_proof_rank_of_W''} that  $R(W'')\leq R(W'') + \mix -2$, thus $\mix=2$ which ends our proof.
\end{proof}

\begin{rem}
        In Lemma \ref{lem_if_rank_3_more_than_dimension} the assumption $\Bis=\emptyset$ can be relaxed in the following way. Let us assume that the pair $(W', W'')$ is such that
        the additivity of ranks does not hold,
        $\Prime=\emptyset, \Bis \neq \emptyset$ and 
        assumptions (i),(iii),(iv) of the Lemma \ref{lem_if_rank_3_more_than_dimension} are fulfilled.
        We digest and replenish obtaining $(\digestion\repletion{W'}, \digestion\repletion{W''})=(W', \digestion\repletion{W''})$. From Corollary~\ref{cor_can_replete_and_digest} follows that $(W', \digestion\repletion{W''})$ is another pair without the rank additivity property and 
        which fulfills all assumptions of Lemma \ref{lem_if_rank_3_more_than_dimension}. 
        Thus we obtain:
        \begin{itemize}
            \item $R(W')+R(\digestion\repletion{W''}) -1= R(W)$
            \item $R(\digestion\repletion{W''})=\hl+\vl+2$
            \item $R(\digestion\repletion{W''})=\hr+\vr+1$
            \item $\mix=2$ 
        \end{itemize}
\end{rem}

If both tensors $p',p''$ fulfill the assumptions of the Lemma~\ref{lem_if_rank_3_more_than_dimension}, then the pair $(p',p'')$ posses the rank additivity property.

\begin{cor}\label{cor_if_p1_p2_the_same_and_rank_3_more_than_dimension}
    Over an algebraically closed base field $\kk$, assume $W'\subseteq \kk^{\bfb'} \otimes \kk^{\bfc'}$ is of dimension $\bfa'$, $W''\subseteq \kk^{\bfb''} \otimes \kk^{\bfc''}$ is of dimension $\bfa''$, the corresponding tensor is $p = p' \oplus p'' \in A' \otimes B' \otimes C' \oplus A'' \otimes B'' \otimes C''$ and:
    \begin{enumerate}
     \item $W', W''$ are concise,
     \item $\Prime=\Bis=\emptyset$,
     \item \label{it:eq_cor_if_rank_3_more_than_dimension}
   $\bfa'+3= R(W')$,
    \item \label{it:eq_cor_if_rank_3_more_than_dimension2}
   $\bfa''+3= R(W'')$,
     \item \label{it:ineq_cor_if_rank_3_more_than_dimension}
   $R(W') \leq \bfb' + \bfc'$,
     \item \label{it:ineq_cor_if_rank_3_more_than_dimension2}
   $R(W'') \leq \bfb'' + \bfc''$.
\end{enumerate}
   Then additivity of ranks holds.
\end{cor}

\begin{proof}
 Assume the additivity of ranks does not hold. We obtain from Lemma~\ref{lem_if_rank_3_more_than_dimension}  that
    $R(W')=\hl+\vl+2$, $R(W'')=\hr+\vr+1$. Now we can exchange every $'$ with $''$ and $(\Prime, \HL, \VL)$ with $(\Bis, \HR, \VR)$ and apply the same lemma again. This time we have (in the notation before the exchange) $R(W')=\hl+\vl+1$, $R(W'')=\hr+\vr+2$. A contradiction.
\end{proof}

We end the subsection with 
a positive answer for the question about rank additivity property for $2\times2$ matrix multiplication tensors (over a base field $\CC$), i.e. $\mu_{2,2,2}\oplus\mu_{2,2,2}\in \CC^{4+4} \otimes \CC^{4+4} \otimes \CC^{4+4}$. 
On a way to prove it we need to show the following fact. 
If both tensors from the pair $(p', p'')$ have ranks less or equal 7, or if
at least one of linear spaces from every triple $\{A',B',C'\}$, $\{A'',B'',C''\}$ is 3 dimensional and all other spaces are at most 4 dimensional,
then the additivity of rank holds.
\begin{cor}\label{cor_additivity_C_443_443}
 Over the base field $\CC$, 
 if 
$(\bfa',\bfb',\bfc')=(4,4,3)$, and either $(\bfa'',\bfb'',\bfc'')=(4,4,3)$ or $(\bfa'',\bfb'',\bfc'')=(4,3,4)$, 
then rank additivity holds.
\end{cor}
\begin{proof}
Assume that the rank additivity does not hold. 
We can reduce the problem to $(W',W'')=(\digestion\repletion W', \digestion\repletion W'')$ by Corollary~\ref{cor_can_replete_and_digest}.
Further, we assume that both tensors are concise. Indeed, if at least one of the tensors is not concise then it follows from either
Theorem \ref{thm_additivity_rank_intro_jarek}~\ref{it:3_thm_additivity_rank_intro_jarek}
or 
Theorem \ref{t:jaja_2_0_hook} that the additivity of rank holds.

\cite[p.~6]{atkinson_stephens_maximal_rank_of_tensors} 
   claims that the maximal rank of tensors from $\CC^4 \otimes \CC^4 \otimes \CC^3$  is $7$ (see also \cite[Prop.~2]{sumi_miyazaki_sakata_maximal_tensor_rank}).
   As a corollary from Theorem~\ref{thm_additivity_rank_intro_jarek}~\ref{it:1_thm_additivity_rank_intro_jarek}, we may restrict ourselves to the case $R(W')=R(W'')=7$.
Applying Corollary~\ref{cor_if_p1_p2_the_same_and_rank_3_more_than_dimension} we obtain a contradiction.
\end{proof}

We prove the following corollary in a similar way.

\begin{cor}\label{cor_additivity_C_444_444_r_7}
 Over the base field $\CC$, if both 
tensors have ranks less or equal 7,
then rank additivity holds.
 
 In particular, over the base field $\CC$, a pair of $2 \times 2$ matrix multiplication tensor has  rank additivity property, i.e.  $R(\mu_{2,2,2} \oplus \mu_{2,2,2})=R(\mu_{2,2,2}) + R(\mu_{2,2,2})$.
\end{cor}
\begin{proof}
Assume the rank additivity does not hold. 
We can restrict ourselves to the case $(W',W'')=(\digestion\repletion W', \digestion\repletion W'')$ by Corollary~\ref{cor_can_replete_and_digest}.

We can assume that both tensors are concise by Lemma \ref{lem_bound_on_rank_for_non_concise_decompositions}.
As a corollary from Theorem~\ref{thm_additivity_rank_intro_jarek}~\ref{it:1_thm_additivity_rank_intro_jarek}, we obtain that each of the numbers $\bfa',\bfa'', \bfb', \bfb'', \bfc', \bfc''$ is less or equal 4. 
It follows from Corollary \ref{cor_additivity_C_443_443} and Theorems  \ref{t:jaja_2_0_hook}, \ref{thm_additivity_rank_intro_jarek}~\ref{it:3_thm_additivity_rank_intro_jarek} 
that $(\bfa',\bfb',\bfc')=(4,4,4)$ 
and
 either $(\bfa'',\bfb'',\bfc'')=(4,4,4)$ or $(\bfa'',\bfb'',\bfc'')=(4,4,3)$ or $(\bfa'',\bfb'',\bfc'')=(4,3,4)$. 
    Applying Theorem~\ref{thm_additivity_rank_intro_jarek}~\ref{it:1_thm_additivity_rank_intro_jarek} again, we obtain $R(W')=R(W'')=7$.
It contradicts Corollary~\ref{cor_if_p1_p2_the_same_and_rank_3_more_than_dimension}.
 
 For the last part of the statement, notice that $\mu_{2,2,2} \in \CC^4 \otimes \CC^4 \otimes \CC^4$ and $R(\mu_{2,2,2})=7$ (see Theorem \ref{t:strassen_multiplication} and Example \ref{ex:mu_222}).
\end{proof}

 The following remark follows from Corollary~\ref{cor_additivity_C_444_444_r_7} and Theorem~\ref{thm_additivity_rank_intro_jarek}~\ref{it:1_thm_additivity_rank_intro_jarek}.
 
\begin{remark}\label{rem_smallest_not_known_case}
Over the base field $\CC$, the minimal case in which the counterexample for the rank additivity can occur is $p' \in \CC^4 \otimes \CC^4 \otimes \CC^4$, $p'' \in \CC^4 \otimes \CC^4 \otimes \CC^3$ such that $R(p')=8$, $R(p'')=7$.
\end{remark}

\clearpage{}

\begingroup
    \fontsize{9pt}{9pt}\selectfont
        \bibliography{referencjeFilipa}
\bibliographystyle{alpha}
\endgroup

\end{document}